\RequirePackage{fix-cm}
\documentclass[smallextended]{svjour3}       % onecolumn (second format)
\smartqed  % flush right qed marks, e.g. at end of proof
\usepackage{amssymb,amsmath,amsfonts,mathrsfs}
\usepackage[utf8]{inputenc}
\usepackage[english]{babel}
\usepackage{graphicx}
\usepackage{color}

\usepackage[noend]{algorithmic}

\newcommand{\ZZ}{\mathbb{Z}}
\newcommand{\QQ}{\mathbb{Q}}
\newcommand{\RR}{\mathbb{R}}

\DeclareMathOperator{\cone}{cone}
\DeclareMathOperator{\conv}{conv.hull}
\DeclareMathOperator{\affh}{affine}
\DeclareMathOperator{\linh}{span}
\DeclareMathOperator{\width}{width}

\DeclareMathOperator{\size}{size}

\DeclareMathOperator{\poly}{poly}

\DeclareMathOperator{\relint}{rel.int}
\DeclareMathOperator{\inter}{int}
\DeclareMathOperator{\dom}{dom}

\DeclareMathOperator{\const}{const}
\DeclareMathOperator{\sgn}{sgn}

\DeclareMathOperator{\bord}{br}
\DeclareMathOperator{\relbr}{rel.br}
\DeclareMathOperator{\vol}{vol}
\DeclareMathOperator{\Half}{H}
\DeclareMathOperator{\El}{E}
\DeclareMathOperator{\MIN}{M}

\newcommand\GenFun[2]{F(#1,#2)}
\newcommand\RMIN[3]{M_{#1}^{#2}(#3)}

\newcommand\CONV[1]{Conv_{#1}}
\newcommand\QC[1]{QConv_{#1}}
\newcommand\SQC[1]{SQConv_{#1}}
\newcommand\QCP[1]{QCPoly_{#1}}
\newcommand\CN[1]{Conic_{#1}}
\newcommand\DCN[1]{DConic_{#1}}
\newcommand\EDCN[1]{EvenDConic_{#1}}
\newcommand\ECN[1]{EvenConic_{#1}}

\newcommand\I[1] {^{(#1)} }

\title{ On the complexity of quasiconvex integer minimization problem }
\author{A.~Yu.~Chirkov, D.~V.~Gribanov, D.~S.~Malyshev, P.~M.~Pardalos, S.~I.~Veselov, N.~Yu.~Zolotykh}

\institute{A.~Yu.~Chirkov \at Lobachevsky State University of Nizhny Novgorod, 23 Gagarina Avenue, Nizhny Novgorod, 603950, Russian Federation\\
\email{chir7@yandex.ru}
\and D.~V.~Gribanov \at Lobachevsky State University of Nizhny Novgorod, 23 Gagarina Avenue, Nizhny Novgorod, 603950, Russian Federation\\
National Research University Higher School of Economics, 25/12 Bolshaja Pecherskaja
Ulitsa, Nizhny Novgorod, 603155, Russian Federation\\
\email{dimitry.gribanov@gmail.com}
\and D.~S.~Malyshev \at National Research University Higher School of Economics, 25/12 Bolshaja Pecherskaja
Ulitsa, Nizhny Novgorod, 603155, Russian Federation\\
\email{dsmalyshev@rambler.ru}
\and P.~M.~Pardalos \at University of Florida, 401 Weil Hall, P.O. Box 116595,
Gainesville, FL 326116595, USA
\at National Research University Higher School of Economics, 25/12 Bolshaja Pecherskaja
Ulitsa, Nizhny Novgorod, 603155, Russian Federation\\
\email{p.m.pardalos@gmail.com}
\and S.~I.~Veselov \at Lobachevsky State University of Nizhny Novgorod, 23 Gagarina Avenue, Nizhny Novgorod, 603950, Russian Federation\\
\email{sergey.veselov@itmm.unn.ru}
\and N.~Yu.~Zolotykh \at Lobachevsky State University of Nizhny Novgorod, 23 Gagarina Avenue, Nizhny Novgorod, 603950, Russian Federation\\
\email{Nikolai.Zolotykh@gmail.com}
}

\begin{document}

\maketitle
\thispagestyle{empty}
\pagestyle{plain}

\begin{abstract}
In this paper, we consider the class of quasiconvex functions and its proper subclass of conic functions. The integer minimization problem of these functions is considered, assuming that the optimized function is defined by the comparison oracle. We will show that there is no a polynomial algorithm on $\log R$ to optimize quasiconvex functions in the ball of radius $R$ using only the comparison oracle. On the other hand, if the optimized function is conic, then we show that there is a polynomial on $\log R$ algorithm (the dimension is fixed). We also present an exponential on the dimension lower bound for the oracle complexity of the conic function integer optimization problem. Additionally, we give examples of known problems that can be polynomially reduced to the minimization problem of functions in our classes.
\end{abstract}

\section{Introduction}

\subsection{Motivation and related papers}

We consider the following minimization problem:
\begin{align}
&f_0(x) \to \min\label{IntroProb}\\
&\begin{cases}
f_i(x) \leq 0 \quad i=1,2,\dots,m\\
x \in \ZZ^n,
\end{cases}\notag
\end{align} where $f_0$ and $f_i$ are quasiconvex functions. Let $D = \{x \in \RR^n: f_i(x) \leq 0 \}$ and $D \subseteq r \cdot B_{\infty}^n$, where $r \cdot B_{\infty}^n$ is the ball of radius $r$ in $\RR^n$ related to the Chebyshev norm. The works of Oertel, Wagner, Weismantel \cite{OWW12} and Dadush, Peikert, Vempala \cite{DPV11} give polynomial on $\log r$ algorithms (the dimension is fixed) to solve the problem in the case, when the set $D$ is equipped by the separating hyperplane oracle. The thesis \cite{DADDIS} gives an $(O(n))^n\poly(\log r)$ algorithm to solve this problem and a good survey on related topics. The paper \cite{OWW12} states that a polynomial on $\log r$ algorithm (if the dimension is fixed) can be simply obtained for the following three oracles: the feasibility oracle, the linear integer optimization oracle, and the separation hyperplane oracle. Moreover, a result of \cite{OWW12} affects mixed integer setting. The paper \cite{CENTER} of Basu and Oertel and the thesis \cite{OERTELDIS} of Oertel give a novel approach in integer convex optimization based on the concept of centerpoints generalized to the integer case. These works additionally give the $\Omega(2^n\, \log r)$ lower bound on the complexity of algorithms that are based on the separating hyperplane oracle. See also the books \cite{50YEARS,SCH} for more detailed survey on integer programming.

The historically first work that gives a polynomial integer programming algorithm in a fixed dimension is the work \cite{LEN} of Lenstra. It considers the mixed integer linear programming problem. Next, Frank and Tardos in \cite{FRTAR} and Kannan in \cite{KAN} improved the complexity bounds from \cite{LEN}. The case, when the constraints are expressed by quasiconvex polynomials, was solved in the work \cite{HEINZ} of Heinz. This result was improved in the work \cite{HILDKOP} of Hildebrand and K\"oppe, see also the survey \cite{KOEPPE} of K\"oppe. The problem of recognizing the quasiconvexity of a given polynomial is NP-complete, due to the paper \cite{NPHARD_CONV}. The paper \cite{KHAPOR} of Khachiyan and Porkolab gives an algorithm for the case, when the constraints are expressed as convex semialgebraic sets. Gaven\v{c}iak et al. give in \cite{GAVKNOP} a comprehensive review on the advances in solving convex integer programs from the last two decades. The paper \cite{EIZ_FIP} of Eisenbrand contains an algorithm for linear integer programming with the best known complexity in terms of the constraints number and an input encoding size.

In our work, we consider only algorithms that are based on the comparison oracle. For any pair of points $x,y \in \dom(f)$, the comparison oracle of a quasiconvex function $f$ determines one of the following two possibilities: $f(x) \leq f(y)$ or $f(x) > f(y)$. Our choice is motivated by the following facts. Firstly, the comparison oracle is simpler to implement than the separating hyperplane oracle. Secondly, we will show in this paper that there is no an algorithm solving the problem \eqref{IntroProb} with the comparison oracle, which is polynomial on $n$ or $\log r$. Due to results of \cite{DADDIS,DPV11,OERTELDIS,OWW12}, the last fact means that the problem with the separation hyperplane oracle can not be polynomially reduced to the problem with the comparison oracle. Finally, it is possible to present general subclasses of the quasiconvex functions class that allow to develop polynomial on $\log r$ algorithms, based on the comparison oracle.

This paper has two aims. The first is revealing new classes of functions that can be effectively optimized in a fixed dimension, including already known and important classes of functions. The second one is establishing exponential on the dimension lower bounds on the oracle complexity for the problem \eqref{IntroProb} with respect to some new classes.

\subsection{Content and results of this article}

In Section 2, we introduce two new classes of functions: conic functions and discrete conic functions. For the class of conic functions, we give several equivalent definitions and show that it includes the classes of strictly quasiconvex functions, convex functions, and quasiconvex polynomials. Discrete conic functions are similar to conic functions, but their domains are discrete sets. We will show that there is no natural extension of any discrete conic function to some conic function and give a criteria for situation, when it is possible.

In Section 3, we give some general tools that are helpful for us to prove lower complexity bounds. Additionally, in this section, we show that the problem \eqref{IntroProb} with respect to the classes of conic functions or discrete conic functions can be polynomially reduced to its unconstrained variant.

In Section 4, we present a very simple $(2r)^n$ lower bound on the comparison-based complexity of \eqref{IntroProb}. After that, we give $\Omega(2^n \log r)$ lower bounds to the problem's \eqref{IntroProb} complexity with respect to the classes of conic functions, discrete conic functions, and their even (symmetric) versions.

Finally, in Section 5, we consider examples of concrete problems that can be formulated as optimization problems involving conic or discrete conic functions and give a polynomial on $\log r$ comparison oracle-based algorithm for the conic function integer minimization problem. There is a way how to minimize a convex continuous function using only the so-called zero-order oracle, that is the oracle computing the function value in any given point. Yudin and Nemirovskii (see \cite[pp.\,342--348]{NEMIR}, \cite{YUDIN}) give a polynomial on the dimension and $\log r$ algorithm for continuous minimization of convex continuous functions using calls to the zero-order oracle. Using the results of Gr\"otschel, Lov\'asz, Schrijver and Yudin, Nemirovskii \cite{GRLOVSCH,YUDIN_EVAL} about the equivalence between week separation and week optimization, we can build a week separation oracle for the sets, like $\{x \in \RR^n : f(x) \leq \epsilon \}$, where $f$ is conic. Due to the results of the thesis \cite{DADDIS}, it leads us to an algorithm with the comparison oracle-based complexity $(O(n))^n\poly(\log r)$. Additionally, the result of \cite{OWW12} leads us to an algorithm with a complexity, polynomial on $\log r$, for the mixed integer variant of the problem.

But, for the best of our knowledge, the approach of Yudin and Nemirovskii \cite{NEMIR,YUDIN} can be applied only for convex functions. Since the class of conic functions is not equivalent to the class of convex functions, then the sequence of results described above can not be applied to the conic function integer minimization problem. To this end, we develop our Lenstra's type algorithm for this problem that is based on ideas from the papers \cite{DADDIS,DPV11,HILDKOP,LEN,OERTELDIS,ICONV_MIN_REV}.

We do not present a polynomial on $\log r$ oracle-based algorithm for minimization of discrete conic functions. The papers \cite{CHIR08,2DIMMIN} present these algorithms for the dimension $2$.

\subsection{Future work and remarks}

In Section 3, we give a polynomial on $\log r$ algorithm for the conic function integer minimization problem.  But the analysis of the algorithm is rough, and it is a good idea to make it more accurate in future works. 

It is an interesting open problem to develop weak separation hyperplane oracle for the class of conic functions. The existence of such algorithm gives opportunity to apply results of from the thesis \cite{DADDIS} of Dadush that give an algorithm with the best known complexity.

Additionally, in this work, we do not present algorithms for integral minimization of discrete conic functions, we only note about algorithms for the dimension $2$ from \cite{CHIR08,2DIMMIN}. The difficulty to build such algorithm for any fixed dimension is the fact that we can ask comparison oracle only in points of some discrete set and the general separation oracle is not helpful in this situation. We are planning to work on these problems in the future. 

We also note that our algorithm can be helpful to design FPT-algorithms for some combinatorial optimization problems. See the papers \cite{FIXEDP,GAVKNOP} for details.

\section{Definitions, notation and some preliminary results}

The $n$-dimensional ball of a radius $r > 0$, centered in a point $y \in \RR^n$ and related to the norm $l_p$ is denoted by $y + r \cdot B_p^n$. In other words,
$$
y + r \cdot B_p^n = \{x \in \RR^n : ||x-y||_p \leq r\}.
$$

For a matrix $B \in \RR^{m \times n}$, $\cone(B) = \{B t : t \in \RR_+^{n} \}$ is the \emph{cone spanned by columns of} $B$, $\conv(B) = \{B t : t \in \RR_+^{n},\, \sum_{i=1}^{n} t_i = 1  \}$ is the \emph{convex hull spanned by columns of} $B$, $\affh(B) = \{B t : t \in \RR^n,\, \sum_{i=1}^{n} t_i = 1\}$ is the \emph{affine hull spanned by columns of} $B$, and $\linh(B) = \{B t : t \in \RR^n \}$ is the \emph{linear hull spanned by columns of} $B$. If $D \subseteq \RR^n$, then the symbol $\linh(D)$ designates the linear hull, based on the points of $D$. The same is true for other types of the hulls.

For a set $D \subseteq \RR^n$, $\inter(D)$ and $\bord(D)$ are the sets of \emph{interior} and \emph{boundary points} of $D$, respectively. The sets of \emph{interior} and \emph{boundary points related to $\affh(D)$} are denoted by $\relint(D)$ and $\relbr(D)$, respectively.

The set of integer values, started from $i$ and ended in $j$, is denoted by $i:j = \{i,i+1,\dots,j\}$. For a vector $x \in \RR^n$, $x_i$ is the $i$-th component of $x$. The interval between points $y,z \in \RR^n$ is denoted by $$[y,z] = \{x = t y + (1-t)z : 0 \leq t \leq 1\}.$$ We will use the symbol $(y,z)$ to define an open interval.  The set $D$ is said to be \emph{convex} if
$\forall x,y \in D$ $\quad[x,y] \subseteq D$. For a function $f$, $\dom(f)$ is the domain of $f$. For any $y \in \dom(f)$, $\Half^\leq_f(y)$ is the set of contour lines for $f$. In other words,
$$\Half^\leq_f(y) = \{x \in \dom(f) :\: f(x) \leq f(y) \}.$$

The sets $\Half^<_f(y)$, $\Half^=_f(y)$ are defined in a similar way. The set of all minimum points of a function $f$ is denoted by $\MIN_{1}(f)$. If it is not defined, we will put $\MIN_{1}({f}) = \emptyset$. Similarly, $\MIN_{i}({f})$ is the set of all $i$-th minimum points of $f$. The set of all minimum points of a function $f$ on a set $D$ is denoted by $\RMIN{1}{D}{f}$. Similarly, $\RMIN{i}{D}{f}$ is the set of all $i$-th minimum points of $f$ on $D$.

Let us consider the set of functions $f : \dom(f) \to \RR$, such that $\dom(f) \subseteq \RR^n$ is convex. A function $f$ is said to be \emph{quasiconvex} if $$\forall x,y \in \dom(f),\,\forall z \in (x,\,y) \quad f(z) \leq \max \{f(x),\,f(y)\}.$$ A function $f$ is said to be \emph{strictly quasiconvex} if $$\forall x,y \in \dom(f),\, \forall z \in (x,\,y) \quad f(z) < \max \{f(x),\,f(y)\}.$$ A function $f$ is said to be \emph{convex} if $$\forall x,y \in \dom(f),\,\forall t \in (0,\,1) \quad f(t x + (1-t) y) \leq t f(x) + (1-t) f(y).$$ We will denote these classes by the symbols $\QC{n}$, $\SQC{n}$ and $\CONV{n}$ respectively. Additionally, we denote by $\QCP{n}$ the class of quasiconvex polynomials of all possible non-zero degrees with real coefficients.

\begin{note}\label{AlternativeDefQC}
Let $T \subseteq \dom(f)$. It is known that the definition of a quasiconvex function is equivalent to the following definition
\[
\forall x \in \conv(T) \quad f(x) \leq \max\limits_{y \in T} f(y),
\]
and the definition of a strictly quasiconvex function is equivalent to following definition
\[
\forall x \in \conv(T) \setminus T \quad f(x) < \max\limits_{y \in T} f(y).
\]
\end{note}

For points $x\I1,\, x\I2,\, \dots,\, x\I{k} \in \RR^n$, the set
\begin{equation}\label{ConeSymbol}
x\I{k} + \cone(x\I{k}-x\I1,\, \dots,\, x\I{k}-x\I{k-1})
\end{equation} is denoted as $\cone(x\I1,\, x\I2,\, \dots,\, x\I{k-1} | x\I{k})$.

\begin{definition}\label{CNDef}
Let $f : \dom(f) \to \RR$, where $\dom(f)$ is convex.

The function $f$ is \emph{conic} if $\forall y,z \in \dom(f)$ and $\forall t \geq 0$, such that $f(y) \leq f(z)$ and $z + t (z-y) \in \dom(f)$, we have $$f(z + t (z-y)) \geq f(z).$$
\end{definition}

\begin{note}\label{CNSubsetQC} Clearly, the class $\CN{n}$ of conic functions is a subclass of the quasiconvex functions class, that is $\CN{n} \subset \QC{n}$. The inclusion is strict, a counterexample is the quasiconvex function $\sgn(x_1)$.
\end{note}

The next theorem gives two additional ways to define the class of conic functions.
\begin{theorem}\label{CNEquivalentDefinitions}
Let $f : \dom(f) \to \RR$, where $\dom(f) \subseteq \RR^n$ is convex. The following definitions are equivalent:
\begin{enumerate}
\item For any pair of points $y,z \in \dom(f)$ and $\forall t \geq 0$, such that $f(y) \leq f(z)$ and $z + t (z-y) \in \dom(f)$, we have $$f(z + t (z-y)) \geq f(z).$$
\item For any set of points $x\I1, x\I2, \dots, x\I{k},\,y \in \dom(f)$, such that $$f(x\I1) \leq f(x\I2) \leq \dots \leq f(x\I{k})\text{ and }$$ $$y \in \cone(x\I1, x\I2, \dots, x\I{k-1} | x\I{k}),$$ the inequality $f(y) \geq f(x\I{k})$ holds.
Furthermore, we can assume that the points $x\I1, x\I2, \dots, x\I{k}$ are in general position, i.e. no hyperplane contains more than $n$ of them.
\item For any $x \in \dom(f)$, the set $\Half^\leq_f(x)$ is convex (which is equivalent to the quasiconvexity of the function $f$) and
$$
\forall x \in \dom(f) \setminus \MIN_{1}({f}) \quad \Half^=_f(x) \subseteq \relbr(\Half^\leq_f(x)).
$$ If the set $\MIN_{1}({f})$ is not defined, we will put it to be empty.
\end{enumerate}

Figure 1 gives an illustration for the first two equivalent definitions.
\end{theorem}

\begin{figure}[h]\label{CNDefFig}
\centering
\includegraphics[clip,scale=0.58]{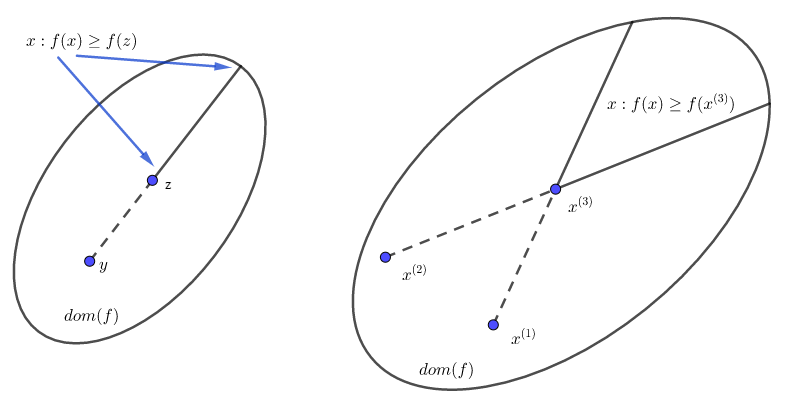}
\caption{An explanation of Theorem \ref{CNEquivalentDefinitions}. Definition 1 on the left and definition 2 on the right.}
\end{figure}

\begin{proof}
{\it The equivalence of 1 and 2.} Any triangulation of the polytope $\conv(x\I1, x\I2, \dots, x\I{k})$ induces a triangulation of the cone $$\cone(x\I1, x\I2, \dots, x\I{k-1} | x\I{k})$$ into simple cones. Thus, it can be assumed that the points $x\I1, x\I2, \dots, x\I{k}$ are in the general position.

Clearly, the first part follows from the second part. We will prove the converse statement. Let the points $x\I1, x\I2, \dots, x\I{k} \in \dom(f)$ be in general position and $\max\limits_{1\leq i \leq k} f(x\I{i}) \leq f(x\I{k})$. Let us fix $y \in \dom(f)\setminus\{x\I{k}\}$, such that $y \in \cone(x\I1, x\I2, \dots, x\I{k-1} | x\I{k})$. We show that the inequality $f(y) \geq f(x\I{k})$ is true. Consider the line $L$ passing through the points $y$ and $x\I{k}$. The line $L$ intersects the set $\conv(x\I1, x\I2, \dots, x\I{k-1})$ in some point $z$. The function $f$ is defined in the point $z$, because $\dom(f)$ is convex. The quasiconvexity of $f$ implies that $f(z) \leq f(x\I{k})$. By this fact and Definition 1, the inequality $f(y) \geq f(x\I{k})$ holds.

{\it The equivalence of 1 and 3.} The implication $1 \to 3$. Suppose that $\exists z \in \dom(f) \setminus \MIN_{1}({f})$, such that $z \in \relint(\Half^\leq_f(z))$. The definition of $z$ implies the existence of a point $v \in \dom(f)$, such that $f(v) < f(z)$. Let $B = \cone(\Half^\leq_f(v)|z)$. Since $z \in \relint(\Half^\leq_f(z))$, then $\exists u \in B \cap \Half^\leq_f(z)$ and $u \not= z$. By Definition 2, we have $f(u) \geq f(z)$, and, therefore, $f(u) = f(z)$. By Note \ref{CNSubsetQC}, the set $\Half^\leq_f(v)$ is convex. Therefore, the ray $\cone(u|z)$ intersects the set $\Half^\leq_f(v)$ in some point. By this fact and Definition 1, $$\forall x \in \cone(u|z) \quad f(x) \geq f(z).$$ The last inequality contradicts to the inequalities $$x \in \Half^\leq_f(v) \quad f(x) \leq f(v) < f(z).$$

The implication $3 \to 1$. Consider points $y,z \in \dom(f)$, such that $f(y) \leq f(z)$. The claim is $$\forall t \geq 0,\text{ such that } x_t = z + t (z-y) \in \dom(f), \quad f(x_t) \geq f(z).$$ If $z \in \MIN_{1}({f})$, then the inequality $f(x) \geq f(z)$ holds for all $x \in \dom(f)$. So, we suppose that $z \notin \MIN_{1}({f})$ and $f(x_t) < f(z)$. If $f(y) < f(z)$, then $f(z) > \max\{f(x_t),f(y)\}$, which contradicts to the quasiconvexity of $f$. Thus, $f(y) = f(z)$. Since $y,z \notin \MIN_{1}({f})$, we have $y,z \in \relbr(\Half^\leq_f(y))$ and $x_t \in \relint(\Half^\leq_f(y))$. Hence, there is a sphere $B=x_t + r \cdot B_{2}^n \cap \affh(\Half^\leq_f(y))$ of some non-zero radius, such that $B \cap \Half^{\leq}_{f}(y) = B$. Let us consider the set $M = \conv(y,B)$. The convexity of the set $\Half^{\leq}_{f}(y)$ implies $M \subseteq \Half^{\leq}_{f}(y)$. The point $z$ is an internal point of the segment $[y,x_t]$, and, therefore, $z$ is included in $M$ with some relative neighborhood. The last fact contradicts to the statement $z \in \relbr(\Half^{\leq}_{f}(y))$.
\end{proof}

The next theorem shows that the class $\CN{n}$ contains some important subclasses.
\begin{theorem}\label{CNInclusions}
The following strict inclusions hold:
\begin{enumerate}
\item $\SQC{n} \subset \CN{n} \subset \QC{n}$,
\item $\QCP{n} \subset \CN{n}$,
\item $\CONV{n} \subset \CN{n}$.
\end{enumerate}
\end{theorem}
\begin{proof}
The inclusion $\CN{n} \subset \QC{n}$ was analyzed in Note \ref{CNSubsetQC}.

Let us prove that $\QCP{n} \subset \CN{n}$. To this end, we consider a polynomial $f \in \QCP{n}$. First of all, we will show that if $z \notin \MIN_{1}({f})$, then the set $\Half^{\leq}_{f}(z)$ is full-dimensional. That is
$\dim(\Half^{\leq}_{f}(z)) = n$. Suppose that it is not true. By the definition of the point $z$, there is a point $y \in \Half^{\leq}_{f}(z)$, such that $f(y) < f(z)$. By the continuity argument, $\forall \epsilon > 0$ there is a ball $B =y + r \cdot B_{2}^n$ with some non-zero radius, such that $$\forall x \in B\quad |f(y)-f(x)| \leq \epsilon.$$ Choosing $\epsilon \leq f(z)-f(y)$ to be small enough, we have $B \subseteq \Half^{\leq}_{f}(z)$. The last inclusion contradicts to the fact that $\dim(B) = n$.

Let us prove that, for any polynomial $f \in \QCP{n}$ and for any point $z \notin \MIN_{1}({f})$, the equality $$\bord(\Half^{\leq}_{f}(z)) = \Half^{=}_{f}(z)$$  holds. The inclusion $\bord(\Half^{\leq}_{f}(z)) \subseteq \Half^{=}_{f}(z)$ follows from the continuity of the polynomial $f$. Let us prove the reverse inclusion. Suppose that $z \in \inter(\Half^{\leq}_{f}(z))$. Note that if $f(x) = \const$ on some $n$-dimensional convex set, then $f(x) \equiv \const$. The last fact contradicts to the definition of the class $\QCP{n}$. There is a ball $B =z + r \cdot B_{2}^n$ with some non-zero radius, such that $B \cap \Half^{\leq}_{f}(z) = B$. Let us choose points $u\I1,u\I2,\dots,u\I{n-1} \in B$, such that they are in the general position and $f(u\I{i}) \not= f(z)$, for any $i \in 1:(n-1)$. If such a choice is not possible, then all of the points $x \in B,\,f(x) \not= f(z)$ are contained in some affine subspace of the dimension strictly less than $n$, and, therefore, $f$ is a constant. Suppose that the choice is possible. Then let us consider the sets $X= \conv(u\I1,u\I2,\dots,u\I{n-1})$ and $Y = \conv(v\I1,v\I2,\dots,v\I{n-1})$, where the points $v\I{i}$ are the symmetry points for $u\I{i}$ with respect to $z$. The following two cases are possible:

1) For all $x \in \conv(z,Y)$, the equality $f(x) = f(z)$ holds. Then $f$ is a constant.

2) There is a point $y \in \conv(z,Y)$, such that $f(y) < f(z)$. Let us consider the line $L$, passing through the points $z,y$. The line $L$ intersects the set $X$ in a point $\hat y$. Since $\hat y \in X$, then, by the quasiconvexity of $f$, we have $f(\hat y) < f(z)$. Hence, $z \in [y,\hat y]$ and $f(z) > \max\{f(y),f(\hat y)\}$, which contradicts to the quasiconvexity of $f$. The inclusion $\QCP{n} \subset \CN{n}$ is strict, a counterexample is the function $f(x) = |x_1|$ that is clearly conic, but it is not a polynomial.

To prove the inclusion $\SQC{n} \subset \CN{n}$, suppose that there are points $y,z \in \dom(f)$, such that $f(y) \leq f(z)$. We need to prove that for $\forall t \geq 0$, such that $x_t = z + (z-y) t \in \dom(f)$, the inequality $f(x_t) \geq f(z)$ holds. Suppose to the contrary that $\exists t > 0$, such that $f(x_t) < f(z)$. The point $z$ is an internal point of the segment $[y,x_t]$. By the definition of the quasiconvexity, we have $f(z) < \max\{f(y), f(x)\}$. If $f(y) \leq f(x)$, then we have $f(z) < f(x)$, and if $f(y) > f(x)$, then we have $f(z) < f(y)$. In both cases we have a contradiction. The inclusion $\SQC{n} \subset \CN{n}$ is strict, because the class $\SQC{n}$ does not contain constants.

To prove the inclusion $\CONV{n} \subset \CN{n}$, suppose that there are points $y,z \in \dom(f)$, such that $f(y) \leq f(z)$. We need to prove that for $\forall t \geq 0$, such that $x_t = z + (z-y) t \in \dom(f)$, the inequality $f(x_t) \geq f(z)$ holds. Since $x_t = z + (z-y) t$, we have $z = \frac{1}{1+t} x_t + \frac{t}{1+t} y$. By the definition of the convexity, we have $$
f(z) \leq \frac{1}{1+t} f(x_t) + \frac{t}{1+t} f(y)
$$ and
$$
f(x_t) \geq (1+t) f(z) - t f(y) \geq f(z).
$$ The inclusion $\CONV{n} \subset \CN{n}$ is strict, a counterexample is any concave, decreasing function, for example $\log x_1$.

\end{proof}

\begin{theorem}\label{CNOperations}
The class $\CN{n}$ is closed with respect to the following operations.
\begin{enumerate}
\item Let $f_i \in \CN{n}$ and $w_i \in \RR_+$, for any $i \in 1:k$. Then the function $g(x) = \max\limits_{i \in 1:k}\{w_i f_i(x)\}$ belongs to the class $\CN{n}$, where $\dom(g) = \bigcap\limits_{i \in 1:k} \dom(f_i)$.

\item Let $f \in \CN{n}$ and $h : \RR \to \RR$ be a conic, non-decreasing function. Then the function $g = h \circ f$ belongs to the class $\CN{n}$.

\item Let $f \in \CN{m}$, $A \in \RR^{m \times n}$ and $b \in \RR^m$. Then the affine image $g(x) = f(A x + b)$ belongs to the class $\CN{n}$.

\item Let $f_1, f_2 \in \CN{n}$ and $D = \dom(f_1) \cap \dom(f_2)$. Then the function $g(x) = \binom{f_1(x)}{f_2(x)} : D \to \RR^2$ is conic with respect to the lexicographical order in $\RR^2$. 
\end{enumerate}
\end{theorem}
\begin{proof}
{\bf Let us proof the proposition 2.} If $h(f(z)) > h(f(y))$, then by the non-decreasing property and the definition of conic functions we have $f(z) > f(y)$ and $f(z + t (z-y)) \geq f(z)$ for any $t \geq 0$, such that $z + t (z-y) \in \dom(f)$.  Hence, $h (f(z + t (z-y))) \geq h(f(z))$. Consider the case $h(f(z)) = h(f(y))$ and two sub-cases $f(z) \geq f(y)$, and $f(z) < f(y)$. The sub-case $f(z) \geq f(y)$ is equivalent to the previous situation. Let $f(z) < f(y)$. Since $h$ is conic, then $h(x) \geq h(f(z)) = h(f(y))$ for any $x \in (-\infty,f(z)] \cup [f(y),+\infty)$. Since $h$ is non-decreasing, we have that $h(x) = h(f(z)) = h(f(y))$ for any $x \in (-\infty,f(y)]$, and consequently the value $f(z)$ is minimum point of $h$ on $\RR$. So, the inequality $h(f(z + t (z-y))) \geq h(f(z))$ trivially holds for any $t \in \RR$, such that $z + t (z-y) \in \dom(f)$.

{\bf Let us show that the function $g(x) = \max\{f_1(x),f_2(x)\}$ is conic if the functions $f$ and $g$ are conic. Together with the proposition 2 it proofs correctness of the proposition 1.} Let $g(z) \geq g(y)$ for some points $y,z \in \dom(f)$. Consider the cases \begin{enumerate}
\item $f_1(z) \leq f_2(z)$, $f_1(y) \leq f_2(y)$;
\item $f_1(z) \leq f_2(z)$, $f_1(y) > f_2(y)$.
\end{enumerate} The remaining cases are symmetric versions of the previous. In the {\bf case 1} we have $f_2(z) = g(z) \geq g(y) = f_2(y)$. Since $f_2$ is conic we have $g(z + t (z-y)) \geq f_2(z + t (z-y)) \geq f_2(z) = g(z)$ for any $t \geq 0$, such that $z + t (z-y) \in \dom(g)$, so $g$ is conic. In the {\bf case 2} we have $f_2(z) = g(z) \geq g(y) = f_1(y) > f_2(y)$. As in the {\bf case 1} we have $g(z + t (z-y)) \geq f_2(z + t (z-y)) \geq f_2(z) = g(z)$.

{\bf Let us proof the proposition 3.} Let $g(z) \geq g(y)$ for some points $y,z \in \dom(g)$, then $f(A z + b) \geq f(A y + b)$. Since $f$ is conic we have $g(z + t (z - y)) = f(A z + t A (z-y) + b) = f(A z + b + t((A z + b) -(A y + b)) ) \geq f(A z + b) = g(z)$, so $g$ is conic.

{\bf Let us proof the proposition 4.} Let $g(z) \succeq_{lex} g(y)$ for some points $y,z \in \dom(g)$. Consider the case $f_1(z) > f_1(y)$. It is easy to see that $f_1(z + t (z-y)) > f_1(z)$ for $t > 0$. Definitely, if $f_1(z + t (z-y)) = f_1(z)$ for $t > 0$, then, by the definition of conic function, we have the contradiction $f_1(y) \geq f_1(z)$. Hence, in the case $f_1(z) > f_1(y)$ we have $g(z + t (z-y)) \succ_{lex} g(z)$ for $t > 0$. Consider the case $f_1(z) = f_1(y)$, so $f_2(z) \geq f_2(y)$. Since $f_2$ is conic we have $f_2(z + t (z-y)) \geq f_2(z)$ for any $t \geq 0$. Hence, in both of the possible cases $f_1(z + t (z-y)) > f_1(z)$ and $f_1(z + t (z-y)) = f_1(z)$ we have $g(z + t (z-y)) \succeq_{lex} g(z)$.
\end{proof}

It is easy to show that the sum of two qusiconvex functions, defined on different domains, is quasiconvex. That is, if $f$ and $g$ are quasiconvex, then the function $h(x,y) = f(x) + g(y)$ is quasiconvex too. For the class $\CN{n}$, this property does not hold, counterexamples are the functions $f(x) = 3 x$ and $g(x) = -2^x$. The function $h(x,y) = 3x - 2^y$ is not conic. To prove that it suffices to consider the points $(0,0)$, $(1,1)$ and the ray, passing through these points. The sum of conic functions, defined on the same domain, can be a non-conic function. Again, counterexamples are the functions $f$ and $g$.

A function $f$ is said to be \emph{even} if $f(x) = f(-x)$, for any $x,(-x) \in \dom(f)$. A set $D \subseteq \RR^n$ is said to be \emph{discrete} if $\forall x \in D$ there is a ball $B =x + r \cdot B_2^n$ with $r > 0$, such that $D \cap B = \{x\}$. 

%A set $D \subseteq \RR^n$ is said to be \emph{uniformly discrete} if there is a ball $B = r \cdot B_2^n$ with $r > 0$, such that $\forall x \in D$ the equality $D \cap (x + B) = \{x\}$ holds.

\begin{definition}\label{DCNDef}
Let $f : \dom(f) \to \RR$, where $\dom(f) \subset \RR^n$ is discrete.

The function $f$ is \emph{discretely conic} if for any points $y,\,x\I1,\, x\I2,\, \dots,\, x\I{k} \in \dom(f)$, such that \[f(x\I1) \leq f(x\I2) \leq \dots \leq f(x\I{k})\text{ and }\] \[y \in \cone(x\I1,\, x\I2,\, \dots,\, x\I{k-1} | x\I{k}),\] the inequality $f(y) \geq f(x\I{k})$ holds.

The class of discretely conic functions will be denoted by $\DCN{n}$.
\end{definition}

\begin{note}\label{ClassDefinitionsNote}
The classes $\CN{n}$ and $\DCN{n}$ contain functions with values in $\RR$. But actually, we can use any totally ordered set instead of $\RR$. For example, it can be the set $\RR^k$ with the lexicographical ordering.
\end{note}

Is it possible to extend any function in $\DCN{n}$ to a function in $\CN{n}$? The theorem~\ref{ExtensionToCNTh} below answers this question.
\begin{definition}\label{ExtensionToCNDef}
Let $f \in \DCN{n}$. The function $g \in \CN{n}$ is \emph{an extension of the function $f$}, if
\[
\dom(g) = \conv(\dom(f))\text{ and}
\]
\[
g(x) = f(x),\text{ for }x \in \dom(f).
\]
\end{definition}

Let us consider a function $f \in \DCN{n}$, such that $\forall \alpha \in \RR$ all the sets $\{x : f(x) \leq \alpha\}$ are finite. Since $\dom(f)$ is discrete and the sets $\{x : f(x) \leq \alpha\}$ are finite, then the sets $\MIN_{i}({f})$ are uniquely defined, for any $i \geq 1$. The sets $\MIN_{i}({f})$ are finite and form the unique partition of $\dom(f)$:
\[
\dom(f) = \bigcup\limits_{i \geq 1} \MIN_{i}({f}).
\] Let $z\I{i}$ be some representative of the set $\MIN_{i}({f})$ for $i \geq 1$.

\begin{theorem}\label{ExtensionToCNTh}
A function $f \in \DCN{n}$ has an extension in terms of Definition \ref{ExtensionToCNDef} if and only if $\forall i \geq 2$ the following inclusion is true:
\[
\MIN_{i}({f}) \subseteq \relbr(P_i),
\] where $P_i = \conv(\MIN_{1}({f}),\MIN_{2}({f}),\dots,\MIN_{i}({f}))$.

Since $P_i = \conv(\Half^{\leq}_{f}(z\I{i}))$, the requirement can be reformulated as follows: for all $z \in \dom(f)\setminus \MIN_{1}({f})$ the following inclusions hold
$$
\Half^{=}_{f}(z) \subseteq \relbr(\, \conv(\Half^{\leq}_{f}(z)) \,).
$$
\end{theorem}
\begin{proof}
Let us show that if $\exists i \geq 2$, such that $\MIN_{i}({f}) \cap \relint(P_i) \not= \emptyset$, then the extension of $f$ does not exist. Suppose to the contrary that there is some extension $g \in \CN{n}$ of the function $f$. By Note \ref{CNSubsetQC}, the function $g$ is quasiconvex, and we have
\begin{equation}\label{ExtensionToCNTh_1}
\forall i \geq 1,\, \forall x \in P_i \quad g(x) \leq g(z\I{i}).
\end{equation}
Without loss of generality we can assume that $z\I{i} \in \MIN_{i}({f}) \cap \relint(P_i)$. Let $U = \cone(P_{i-1}|z\I{i}) \cap P_i$. Since $g \in \CN{n}$, then $g(x) \geq g(z\I{i})$ for all $x \in U$. By the inequalities \eqref{ExtensionToCNTh_1}, we have that $U \subseteq \Half^{=}_{g}(z\I{i})$. Additionally, $U \not= \emptyset$, because $z\I{i} \in \relint(P_i)$. Suppose that $u \in U$ and the ray $R = \cone(u|z\I{i})$ intersects the set $P_{i-1}$ in some point. The last contradicts to the fact that $$\forall x \in R \quad g(x) \geq g(z\I{i}).$$

Let us show that if the conditions of the theorem are true, then an extension of $f$ exists. The function $g$ is built by inductive propagation of its values to the sets $P_i$, for any $i \geq 1$. To this end, we introduce an additional notation $B_i = P_i \setminus P_{i-1}$.

Let $g(x) \equiv f(z\I1)$, for all $x \in P_1$. Assuming that $g$ has already been defined for the set $P_{i-1}$, we show how to extend $g(x)$ to the set $P_i$. Let $g(x) = f(z\I{i})$, for all $x \in \relbr(B_i)$. For any $t \geq 0$, we consider the sets
\begin{equation}\label{ExtensionToCNTh_2}
U(t) = \relint(B_i) \cap \{x : d(x,P_{i-1}) = t \},
\end{equation} where $d(x,P_{i-1})$ is the $l_2$-distance from the point $x$ to the convex set $P_{i-1}$. The sets $U(t)$ are subsets of the set $\relint(B_i)$ that have an equal distance to the boundary of $P_{i-1}$. Let
\begin{equation}\label{ExtensionToCNTh_3}
\tau = \sup\limits_{x \in B_{i}} d(x,P_{i-1}),
\end{equation} then
\begin{equation}\label{ExtensionToCNTh_4}
\relint(B_i) = \bigcup\limits_{0 < t < \tau} U(t).
\end{equation}

The following formula extends the function $g$ to the sets $U(t)$, for any $0 < t < \tau$:
\begin{equation}\label{ExtensionToCNTh_5}
g(x) \equiv \frac{t f(z\I{i}) + (\tau-t)f(z\I{i-1})}{\tau} \text{ for any}~x \in U(t).
\end{equation} Then, the formula \eqref{ExtensionToCNTh_4} gives the extension of $g$ to the set $P_i$.

We show by induction that $g_i = g|_{P_i}$ is contained in the class $\CN{n}$ for any $i \geq 1$. Trivially, $g_1 \in \CN{n}$, because $g_1 \equiv const$. Let $g_{i-1} \in \CN{n}$. We need to show that $g_i \in \CN{n}$. The claim is: $\forall y,z \in P_i$, $g(y) \leq g(z)$ and $y \not= z$ we have
$$\forall x \in \cone(y|z) \cap P_i \quad g(x) \geq g(z).$$
If $\cone(y|z) \cap P_i \subseteq P_{i-1}$, then the claim follows from the inductive assumption. In the opposite case, we have $\cone(y|z) \cap B_i \not= \emptyset$. There are the only three possible cases: 1)~$y,z \in P_{i-1}$; 2)~$y \in P_{i-1},\, z\in B_i$; 3)~$y,z \in B_{i}$.

\begin{description}
\item[Case 1: $y,z \in P_{i-1}.$] The following equality holds: $$\cone(y|z) \cap P_i = [z,v] \cup (v,u],$$ where $[z,v] \subseteq P_{i-1}$, $(v,u] \subseteq B_i$, $v \in \relbr(P_{i-1})$ and $u \in \relbr(B_i)$. By the inductive assumption, the claim is true for $x \in [z,v]$. By the definition, values of $g$ in the segment $(v,u]$ are strongly greater than in the segment $[z,v]$. So, we need to show that values of $g$ are not increasing along the segment $(v,u]$. The distance $d(x,\relbr(P_{i-1}))$ is not decreasing along the segment $(v,u)$. By formulae \eqref{ExtensionToCNTh_2}~-~\eqref{ExtensionToCNTh_5} $g(x)$ is not decreasing too. The value $g(u)$ is maximal, because $u \in \relbr(B_i)$.

\item[Case 2: $y \in P_{i-1},\, z\in B_i.$]  The segment $[y,z]$ must intersect the segment $\relbr(P_{i-1})$ in some point $v$. By the same reasons, values of $g(x)$ are not decreasing along the ray $\cone(v|z)$.

\item[Case 3: $y,z \in B_{i}.$] The case $z \in \relbr(B_i)$ is trivial, because then the intersection of $B_i$ and $\cone(y|z)$ consists of only one point $z$. Let $y \in \delta(B_i)$. By construction, the inequality $g(y) \leq g(z)$ is only possible in the case, when $z \in \relbr(B_i)$. Let $y,z \in \relbr(B_i)$. Let us consider the set $$U_ z = \{x \in P_i  : d(x,P_{i-1}) \leq d(z,P_{i-1}) \}.$$ By definition, $$\relbr(U_z) = \{x \in P_i  : d(x,P_{i-1}) = d(z,P_{i-1}) \}.$$ Let $d(v,P_{i-1}) < d(z,P_{i-1})$ for some point $v \in \cone(y|z) \cap P_i$, then $v \in \relint(U_z)$. Hence, there is a neighborhood $B =v + r \cdot B_{2}^n$ of the point $v$, such that $B \cap U_z = B$. Since the set $U_z$ is convex, then $\conv(y,B) \subseteq U_z$. But the point $z$ is in $\conv(y,B)$ with some neighborhood. It contradicts to the statement $z \in \relbr(U_z)$. Thus, $d(x,P_{i-1}) \geq d(z,P_{i-1})$ for $\forall x \in \cone(y|z) \cap P_i$, which meets the corresponding inequalities for the function $g$.
\end{description}
\end{proof}

\begin{corollary}\label{ExtensionToCNCorr}
For any function $f \in \DCN{n}$, there is a function $g \in \CN{n}$, such that
\[
\dom(g) = \conv(\dom(f)),
\]
\[
\MIN_{1}({g}) = \MIN_{1}({f}),\text{ and}
\]
\[
\emptyset \not= \MIN_{2}({g}) \subseteq \MIN_{2}({f}).
\]
\end{corollary}

\begin{note}\label{ExtensionToCNNote}
It is not hard to see that if it is possible to extend a function $f \in \DCN{n}$ to a function $g \in \CN{n}$, then the function $f$ can be extended to any convex set $M$, such that $$\conv(\dom(f)) \subseteq M.$$ To see this, we can use the scaled distance to the convex set $\conv(\dom(f))$.
\end{note}

\section{General notes about the conic function minimization problem}

Let $D$ be a discrete set and $\mathscr{F}$ be some class of functions.
We define the notion of the \emph{generalized discrete minimization functional} $\GenFun{\mathscr{F}}{D} : \mathscr{F} \to \RR$ as follows.

\begin{definition}
Let $f \in \mathscr{F}$ and $D \subseteq \dom(f)$. The functional $\GenFun{\mathscr{F}}{D}$ is determined by the following equality:
\begin{equation}\label{GeneralMinFunctionalDef}
\GenFun{\mathscr{F}}{D}(f) = \min\limits_{x \in D} f(x).
\end{equation}
If $F = \GenFun{\mathscr{F}}{D}$, then the set $\mathscr{F}$ will also be denoted by the symbol $\dom(F)$.
\end{definition}

\begin{note}
The functional \eqref{GeneralMinFunctionalDef} defines some minimization problem. By this reason, we will simply call functionals of the type \eqref{GeneralMinFunctionalDef} as minimization problems.
\end{note}

To define functions, we will use the \emph{comparison oracle}. For any pair of points $x,y \in \dom(f)$, the oracle checks whether the inequality $f(x) \leq f(y)$ holds or not.

Let functions $f \in \CN{n}$ and $g_i \in \CN{n}$, for any $i \in 1:m$, be defined by their comparison oracles, and $D$ be some discrete set. We consider the following constraint minimization problem:
\begin{align}
&f(x) \to \min\label{ConditionalMinProblem}\\
&\begin{cases}
g_i(x) \leq 0,\text{ for }i \in 1:m \\
x \in D.
\end{cases}\notag
\end{align}

Let us show that the problem \eqref{ConditionalMinProblem} can be reduced to an unconditional minimization problem in the class $\CN{n}$.
We consider the functions $$t(x) = \max\limits_{i \in 1:m}\{g_i(x)\}$$ and $h(x) = (\, (t(x))_+,f(x) \,)$, where \[(x)_+ = \begin{cases}
x,\text{ for } x \geq 0,\\
0,\text{ for } x < 0.
\end{cases}\]
It is easy to see that the optimal points set of the problem \eqref{ConditionalMinProblem} coincides with the lexicographical minima set of the problem $\GenFun{\CN{n}}{D}(h)$. By properties of functions from the class $\CN{n}$, we have $g \in \CN{n}$. Having comparison oracles of the functions $g_i$ and $f$, we can easily construct a lexicographical comparison oracle for the function $h$. In an alternative variant of the reduction, we can choose the function $h$ in the following way: $h(x) = \max\{f(x), M \cdot (t(x))_+\}$, where $M > 0$ is a sufficiently large constant. Usually, it is easy to choose a value for the constant $M$. For example, if all the functions $g_i$ have integral values, then we can put $M = f(x_0)$ for some point $x_0 \in D$.

\begin{definition}
\emph{An algorithm to solve the minimization problem $F = \GenFun{\mathscr{F}}{D}$} is an algorithm, whose atomic operation is a call to the comparison oracle. The input of such algorithm is the comparison oracle for some function from the $\mathscr{F}$ class. The output of the algorithm is some point from the set $\RMIN{1}{D}{f}$.
\end{definition}

\begin{definition} Let $f \in \dom(F)$. The symbol $\tau_F(A,f)$ denotes the number of oracle calls that an algorithm $A$ takes to solve the problem $F(f)$. Let $$\tau_F(A) = \sup\limits_{f \in \dom(F)} \tau_F(A,f),\text{ and }$$ $$\tau_F = \inf\limits_{A \in \mathscr{A}} \tau_F(A),$$ where $\mathscr{A}$ is the set of all algorithms that solve the problem $F$. The symbol $\tau_F$ denotes \emph{the complexity of the problem $F$}.
\end{definition}

\begin{definition} Any algorithm $A$ for the problem $F$ can be represented by a \emph{binary solution tree}, which is said to be \emph{algorithm's protocol or its program}. Internal nodes of the protocol correspond to oracle calls. Each internal node has exactly two children, the first corresponds to the answer ``yes'' and the second to the answer ``no''. Each path from the root to a leaf corresponds to some concrete way of computations, where an input is a comparison oracle for some $f \in \dom(F)$. Finally, leaves are marked by optimal solutions of the corresponding problem. It is not hard to see that the value $\tau_F(A)$ coincides with maximal length of paths from the root to leaves of the protocol $A$.
\end{definition}

Following \cite{ZOLCHIR}, let us define the notion of a resolving set. It is said that functions $f,g$ have \emph{an equivalent order on points of a set $R$} if $\forall x,y \in R$ the inequality $f(x) \leq f(y)$ holds if and only if the inequality $g(x) \leq g(y)$ holds.
\begin{definition}
Let $F = \GenFun{\mathscr{F}}{D}$ and $f \in \mathscr{F}$. A set $R_f \subseteq \dom(f)$ is a \emph{resolving set} for the function $f$ with respect to the functional $F$ if for any function $g \in \mathscr{F}$, such that $R_f \subseteq \dom(g)$, the following statement holds:
$$
g\text{ and }f\text{ have an equivalent order on points of}~R_f \; \Longrightarrow \; \RMIN{1}{D}{f} \cap \RMIN{1}{D}{g} \not= \emptyset.
$$
\end{definition}

The next lemma shows the importance of resolving sets, a proof easily follows from the definition.

\begin{lemma}\label{ResolvingSetProperty}
Let $F = \GenFun{\mathscr{F}}{D}$, $A$ be a minimization algorithm of the problem $F$ and $f \in \dom(F)$. Let $p$ be the path from the root to a leaf in the protocol $A$ that corresponds to the function $f$. Let $V(p) \subseteq \dom(f)$ be the set of points, in which the oracle calls were asked along the path $p$. Then the set $V(p)$ is resolving for the function $f$.
\end{lemma}

\begin{definition}\label{NonSingularFunctionDef} The function $f \in \mathscr{F}$ is \emph{non-singular with respect to the problem} $F = \GenFun{\mathscr{F}}{D}$ if for any resolving set $R_f$ for $f$ $$R_f \cap \RMIN{1}{D}{f} \not= \emptyset.$$
\end{definition}

The following theorem gives a non-singularity criteria for the classes $\CN{n}$ and $\DCN{n}$.
\begin{theorem}\label{NonSingularFunctionCriteria}
Let $D \subseteq \RR^n$ be some bounded discrete set and the minimization problem be defined by the functional $F = \GenFun{\CN{n}}{D}$ or by the functional $F = \GenFun{\DCN{n}}{D}$. Let, additionally, $Y = \RMIN{1}{D}{f}$ and $Z = \RMIN{2}{D}{f}$. Then a function $f  \in \dom(F)$ is non-singular if and only if
for any subset $T \subseteq Y \cup Z$, such that the points of $T$ are in general position, and $\forall y \in Y$ we have
\begin{equation}\label{NonSingularFunctionCriteriaCondition}
\cone(T|y) \cap Z = \emptyset.
\end{equation}
\end{theorem}

\begin{proof} {\it Sufficiency.} Firstly, let us consider the functional $F = \GenFun{\DCN{n}}{D}$. Let $y \in Y$ and $z \in Z$. Let us define the function $g : D \to \RR$ as follows:
\[
g(x) = \begin{cases}
f(x),\text{ for }x\notin Z\\
\delta,\text{ for }x \in Z,
\end{cases}
\] where $\delta < f(y)$.

Let $R_f$ be a resolving set for $f$ with respect to the problem $F$. Suppose to the contrary that $Y \cap R_f \not= \emptyset$. Clearly, $f$ and $g$ have an equivalent order on points of the set $R_f$, and
$$\RMIN{1}{D}{f} \cap \RMIN{1}{D}{g} = \emptyset.$$ To obtain a contradiction, we need to show that $g \in \DCN{n}$. In other words, all the conditions from Definition \ref{DCNDef} are satisfied. Let $C = \cone(T|p)$ for $T \subseteq \dom(g)$ and $p \in \dom(g)$. We consider different cases to choose the apex $p$ of the cone $C$:
\begin{description}
\item[Case 1: $p \in \dom(g) \setminus (Y \cup Z).$] The conditions from Definition \ref{DCNDef} for $f$ are satisfied on $C$, since $f \in \DCN{n}$. In this case, $C$ is the point set of the third and all the next minima. The values of $C$ have not been changed for $g$. Therefore, the conditions for $g$ are satisfied on $C$.
\item[Case 2: $p \in Z.$] The conditions from Definition \ref{DCNDef} for $f$ are satisfied on $C$, since $f \in \DCN{n}$. The values of $g$ have been changed only in the point $z$, so the conditions for $g$ are satisfied on $C$.
\item[Case 3: $p \in Y.$] In this case, the cone $C$ is based on points of the set $Y \cup Z$ with the apex $p \in Y$. The conditions from Definition \ref{DCNDef} for $g$ can be unsatisfied on $C$ only on points with values less than $g(y)$. The last observation is true only for points of the set $Z$, but the theorem's condition \eqref{NonSingularFunctionCriteriaCondition} states that $C \cap Z = \emptyset$.
\end{description}

By Corollary \ref{ExtensionToCNCorr} of Theorem \ref{ExtensionToCNTh}, we can expand the function $g$ to the function $\hat g \in \CN{n}$, such that
\[
\dom(\hat g) = \conv(\dom(g)),
\]
\[
\RMIN{1}{D}{\hat g} = \MIN_{1}({g}) = Z.
\] The last fact gives the sufficiency condition for the functional $F = \GenFun{\CN{n}}{D}$.

{\it Necessity.} Suppose that the theorem's condition \eqref{NonSingularFunctionCriteriaCondition} is not satisfied. The claim is to construct a resolving set $R_f$ for $f$ with the property $R_f \cap Y = \emptyset$. Let $R_f = D \setminus Y$. We consider the function $g \in \dom(F)$, such that $f$ and $g$ have an equivalent order on points of the set $R_f$. Since the opposite condition of \eqref{NonSingularFunctionCriteriaCondition} holds, then there is a cone $C$, composed from points of the set $Y \cup Z$ with an apex from $Y$, such that $z \in C$ for some $z \in Z$. Suppose that $g(y) > g(x)$ for any $x \in Z$. Then, by Definitions \ref{CNDef} and \ref{DCNDef} of the classes $\CN{n}$ and $\DCN{n}$, we have $g(z) \geq g(y) > g(x)$ for any $x \in Z$. The last observation is a contradiction, because $z \in Z$. Hence, we have $g(y) \leq g(x)$, for some $x \in Z$. The last fact means that
$$\RMIN{1}{D}{f} = Y \cap \RMIN{1}{D}{g} \not= \emptyset.$$ Therefore, $R_f$ is a resolving set for $f$ with the property $R_f \cap Y = \emptyset$. Hence, the function $f$ is singular.
\end{proof}

The next corollary gives a simplified condition of the non-singularity.
\begin{corollary}\label{NonSingularFunctionCriteriaCorr}
Let $F = \GenFun{\CN{n}}{D}$ or $F = \GenFun{\DCN{n}}{D}$, and $f \in \dom(F)$. If $|\RMIN{1}{D}{f}| = |\RMIN{2}{D}{f}| = 1,$ or, in other words, the function $f$ has unique points of the first and second minima on $D$, then $f$ is non-singular with respect to the functional $F$.
\end{corollary}

The following two lemmas are key lemmas to prove lower complexity bounds, which will be presented in this work.
\begin{lemma}\label{LowerBoundLm}
Let $F = \GenFun{\mathscr{F}}{D}$. Let $\mathscr{T} = \{T_i\}$ and $\mathscr{G} = \{f_i\}$ be finite sequences of sets and functions, such that $T_i \subseteq \dom(f_i)$ and $f_i \in \dom(F)$ for any $i \in 1:|\mathscr{T}|$.
Let, additionally, the following minimality condition holds for any set $T_i$:
$$
R\text{ is a resolving set for }f_i \quad\Longrightarrow\quad T_i \subseteq R.
$$

Then, $\tau_F \geq \log_2 |\mathscr{T}|$.
\end{lemma}

\begin{proof}
Let us show that all the functions from $\mathscr{G}$ are distinct. Indeed, if a pair $f_i,f_j \in \mathscr{G}$ coincides for $i \not= j$, then their resolving sets $T_i$ and $T_j$ will coincide by the minimality condition.

Let us consider an oracle algorithm $A$ to solve the problem $F$. We are going to show the existence of an injective map $\phi : \mathscr{T} \to P(A)$, where $P(A)$ is the set of all paths from the root to leaves of the algorithm $A$. Then the resulting estimate $\tau_F \geq \log_2 |\mathscr{T}|$ directly follows from the binarity property of $A$.

Let us consider some function $f_i \in \mathscr{G}$. By Lemma \ref{ResolvingSetProperty}, the set $V(p_i)$ of all points that the algorithm meats along the path $p_i \in P(A)$ on the input $f_i$ is resolving for $f_i$. By the minimality condition, $T_i \subseteq V(p_i)$. After a mapping of each set $T_i$ and each function $f_i$ to a path $p_i$ we have the resulting function $\phi$, which is, possibly, not injective. Let us show the existence of an injective map of the same type. Suppose to the contrary that it does not exist. Then there are sets $T_i,T_j \in \mathscr{T}$, for some $i \not = j$, and a path $p \in P(A)$, such that $T_i \subseteq V(p)$ and $T_j \subseteq V(p)$. Moreover, there are no other paths $\hat p \in P(A)$ with the property $T_i \subseteq V(\hat p)$ or $T_j \subseteq V(\hat p)$. The last observation contradicts to the binarity property of $A$.
\end{proof}

%{ \bf ������. } ������ ������ ����������� ������ �������� ����� ������ � ������ ��������� ������������� �����������. � ������, ����� ��������� ���������� ��� ���� ������������� ����� $a$ � $b$, ����� �� ������ ������, ��� $\text{���}(a,\, b)$ ��������� � ��������� ��������� ������:
%
%\begin{align} \label{init_program}
%&\min|a x + b y| \\
%&\begin{cases}
%|a x + b y| \not= 0 \\
%x,y \in \ZZ \\
%\end{cases} \notag
%\end{align}
%
%��� ���, �������������� ������� �������� ������������, �� ���������� ������ ����� �������� ���������:
%
%\begin{align} \label{main_program}
%&\min|a x - b y| \\
%&\begin{cases}
%|a x - b y| \not= 0 \\
%x,y \in \mathbb{Z_{+}} \\
%\end{cases} \notag
%\end{align}
%
%�������� ������, ��� �������������� ���������� ����������� ������ $\mathscr{F}_2$. ���������� ������������ ����� ������ \ref{init_program} ������������ ����� ������������ ����. ����� ������ ��������, ��� ������� ���������� � ������� $0 \leq x \leq b,\: 0 \leq y \leq a$.

Let $T \subseteq \RR^n$, and
%$C_f(T)$ be the union of all cones $\cone(M|z)$, where $z \in T$, $M \subseteq T$ and $\max\limits_{x \in M} f(x) \leq f(z)$. In other words
\begin{equation}\label{CNHullDef}
C_f(T) = \bigcup\{\cone(M|z): M \cup \{z\} \subseteq T,\, \max\limits_{x \in M} f(x) \leq f(z) \}.
\end{equation} We assume that if $|T| \leq 1$, then $C_f(T) = \emptyset$.

\begin{lemma}\label{SuccessiveMinAnalogLm}
Let $D \subseteq \RR^n$ be a bounded discrete set, $F = \GenFun{\CN{n}}{D}$, $R$ be a resolving set for a function $f \in \dom(f)$, and
\[
Z = \arg\min\{f(x) : x \in D \setminus C_f(R)\}.
\]

If $D \setminus C_f(R) \not= \emptyset$, then $Z \cap R \not= \emptyset$.
\end{lemma}
\begin{proof}
Suppose to the contrary that $Z \cap R = \emptyset$. Let us define the function $g : R \cup \{z\} \to \RR$ as follows:
$$
g(x) = \begin{cases}
\delta,\text{ for }x = z\\
f(x),\text{ for }x \not= z,\\
\end{cases}
$$ where $\delta < \min\{f(x) : x \in D\}$. We are going to show that $g \in \DCN{n}$. Assume that $T \cup \{p\} \subseteq R \cup \{z\}$ and $g(x) \leq g(p)$ for $x \in T$. Additionally, assume that all the points in $T \cup \{p\}$ are in general position. The claim is to show that the conditions from Definition \ref{DCNDef} of the class $\DCN{n}$ are satisfied for any $T$ and $p$. In other words, for any $x \in C = \cone(T|p) \cap \dom(g)$, we need to show that $g(x) \geq g(p)$. We consider the following possible cases:
\begin{description}
\item[Case 1: $f(p) > f(z)$ or $z \notin C \cup T$.] Since $f \in \CN{n}$, the conditions from Definition \ref{DCNDef} are satisfied for $f$. So, we have $z \notin C$ in both cases, and values of the functions $f$ and $g$ coincide on $C$. Therefore, the conditions are satisfied for $g$ on $C$ too.
\item[Case 2: $f(p) \leq f(z),\, z \in C$.] If $p = z$, then we do not have any restrictions on $C$, because $g(z)$ is the minimal value of the function $g$ on $D$. If $p \not= z$, then the case is not possible by the definition of $Z$.
\item[Case 3: $z \in T$.] Values of the functions $f$ and $g$ coincide on $C$. The conditions from Definition \ref{DCNDef} are satisfied for $g$ on $C$, because they are already satisfied for $f$ due to the inclusion $f \in \CN{n}$.
\end{description}

Now, we are going to show that the function $g$ can be extended to the function $\hat g \in \CN{n}$, such that $$\dom\hat g = \conv(\dom g),$$ $$\hat g(x) = g(x) \text{ for } x \in \dom(g).$$ By Theorem \ref{ExtensionToCNTh}, it is possible if and only if $\forall x \in \dom(g)\setminus \{z\}$
$$
\Half^{=}_{g}(x) \subseteq \relbr(\,\conv(\Half^{\leq}_{g}(x))\,).
$$
Since $f \in \CN{n}$, then, by Theorem \ref{ExtensionToCNTh}, the last conditions are satisfied for the points $x \in \dom(g),\, f(x) \geq f(z)$. Suppose for the sake of contradiction that $\exists y \in \dom(g)$, such that $f(y) < f(z)$ and
$y \in \relint(\,\conv(\Half^{\leq}_{g}(y))\,)$.

Since $f \in \CN{n}$, the last inclusion is not possible for $f$. Hence, $$y \in \relint(\conv(P,z))$$ for some subset $P \subseteq \Half^{\leq}_{f}(y)$. Therefore, $z \in \cone(P|y)$ that contradicts to the definition of $Z$.

Finally, we have the pair of functions $f,\hat g \in \CN{n}$, such that $f(x) = \hat g(x)$ for any $x \in R$, but $\RMIN{1}{D}{f} \cap \RMIN{1}{D}{\hat g} = \emptyset$. The last statement contradicts to the fact that $R$ is resolving for $f$.
\end{proof}

\section{Lower bounds of oracle based complexity}

In this section, we give lower comparison oracle-based complexity bounds for the following optimization problems: minimization of a quasiconvex function on any discrete set, minimization of a conic function on the set $r \cdot B_{\infty}^n \cap \ZZ^n$, minimization of an even conic function on the set $r \cdot B_{\infty}^n \cap \ZZ^n \setminus \{0\}$. The same bounds hold for the minimization problems of discrete conic functions and even discrete conic functions. The classes of even conic functions and discrete even conic functions are denoted by the symbols $\ECN{n}$ and $\EDCN{n}$ respectively.

The result and a proof of the following theorem may have already been known. But, we present a proof, because we can not give a correct reference and we want to make the presentation more complete.

\begin{theorem}
Let $M,D \subseteq \RR^n$ be a convex set and a discrete set, respectively, and $F = \GenFun{\QC{n}}{M \cap D}$.
Then $\tau_F \geq |M \cap D|-1$.
\end{theorem}
\begin{proof}
Let us consider the quasiconvex function $f_z : M \to \RR$ that is equal to $1$ everywhere, except the point $z \in M \cap D$, where the function is equal to $0$. Let $\mathscr{F}$ be the set of such functions. Clearly, $|\mathscr{F}| = |M \cap D|$. Any call of the comparison oracle for points from the set $M \cap D$ separates the set $\mathscr{F}$ into two subsets: the first has the size 1, and the second one has $|M \cap D|-1$ elements. Hence, we need at least $|M \cap D|-1$ oracle calls. Oracle calls in points of the set $M \setminus D$ do not give any information about optimal points.
\end{proof}

The last theorem gives that it is needed $(2 \lfloor r \rfloor + 1)^n - 1$ oracle calls to minimize a quasiconvex function in the set $r \cdot B_{\infty}^n$. Hence, it is not possible to build an oracle-based minimization algorithm with a polynomial on $n$ and $r$ complexity.

\subsection{Lower bounds for the class $\CN{n}$}

Let $r \geq 1$ and $F$ denote the functional $\GenFun{\CN{n}}{r \cdot B_{\infty}^n \cap \ZZ^n}$ throughout this subsection.

We introduce a finite family $\mathscr{T}_{n,r}$ of sets $T \subseteq r \cdot B_{\infty}^n \cap \ZZ^n$ and a finite family $\mathscr{H}_{n,r}$ of functions $h_T : T \to 0:(3^{n}-1)$. For any $T \in \mathscr{T}_{n,r}$, the function $h_T$ is a bijection between $T$ and $0:(3^{n}-1)$. The family $\mathscr{T}_{1,r}$ contains all $2 r - 1$ possible sets of the type $T = \{i-1,i,i+1\}$ for any $|i| < r$. If $T = \{i-1,i,i+1\}$, then we put $h(i) = 0$, $h_T(i-1) = 1$ , and $h_T(i+1) = 2$. All possible functions $h_T$, defined in this way, form the family $\mathscr{H}_{1,r}$.

The family $\mathscr{T}_{n,r}$ is obtained from the family $\mathscr{T}_{n-1,r}$ in the following way. Let $T[c]$ be the set that is obtained from $T$ by adding a new coordinate with the value $c$ to each element of $T$. In other words, $T[c] = \{ (x,c) : x \in T\}$. For any integral $i$, satisfying to the inequality $|i|<r$, and for any triplet $(T_1,T_2,T_3) \in \mathscr{T}^3_{n-1,r}$, we construct a new triplet $(T_1[i-1], T_2[i], T_3[i+1])$ and put $T = T_1[i-1] \cup T_2[i] \cup T_3[i+1]$. All possible sets $T$ that can be obtained in this way form the family $\mathscr{T}_{n,r}$. More formally,

\begin{equation}\label{TSetDef}
\mathscr{T}_{n,r} = \bigcup\limits_{i = -r+1}^{r-1} \{T_1[i-1] \cup T_2[i] \cup T_3[i+1] : \text{ for } (T_1,T_2,T_3) \in \mathscr{T}^3_{n-1,r}\},
\end{equation}
where $T[c] = \{ (x,c) : x \in T\}$.

\begin{figure}[h]\label{TnrExampleFig}
\centering
\includegraphics[clip,scale=0.7]{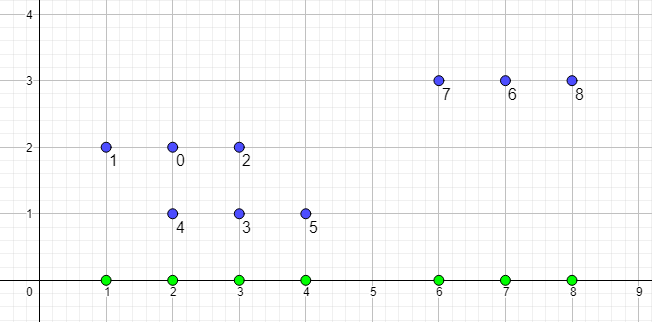}
\caption{The element $T=T_1[i-1] \cup T_2[i] \cup T_3[i+1]$ (blue points) of the family $\mathscr{T}_{2,r}$ built by the triplet $(T_1,T_2,T_3) \in \mathscr{T}^3_{1,r}$, where $T_1 = \{2,3,4\}$, $T_2 = \{1,2,3\}$, $T_3 = \{6,7,8\}$ (green points). The values of the function $h_T \in \mathscr{H}_{2,r}$ are drawn next to the blue points of $T$.}
\end{figure}

For any $T \in \mathscr{T}_{n,r}$, the function $h_T : T \to \RR$ of the class $\mathscr{H}_{n,r}$ is defined in the following way. Due to the formula \eqref{TSetDef}, we have $T = (T_1[i-1],T_2[i],T_3[i+1])$, for some triplet $(T_1,T_2,T_3) \in \mathscr{T}^3_{n-1,r}$, and some value $i$, satisfying the inequality $|i| < r$. Then
\begin{equation}\label{HSetDef}
h_T(y) = \begin{cases}
3^{n-1} + h_{T_1}(x),\text{ for }y=(x,i-1),\\
h_{T_2}(x),\text{ for }y=(x,i),\\
2 \cdot 3^{n-1} + h_{T_3}(x),\text{ for }y=(x,i+1),\\
\end{cases}
\end{equation}
where the functions $h_{T_k} \in \mathscr{H}_{n-1,r}$, for any $k \in 1:3$, are defined inductively in the same way. Figure 2 gives an example of a set $T \in \mathscr{T}_{2,r}$ and a function $h_T \in \mathscr{H}_{2,r}$ defined on this set.

Let us consider a set $T \in \mathscr{T}_{n,r}$ and the function $h = h_T \in \mathscr{H}_{n,r}$. Let the sequence $x\I1,x\I2,\dots,x\I{|T|}$ be formed by the points of $T$, sorted in increasing order of values of the function $h_T$ on them. It follows from definition that the sequence $x\I1,x\I2,\dots,x\I{|T|}$ has the following property:
\begin{equation}\label{THProperty}
x\I{i} \notin C_f(x\I1,x\I2,\dots,x\I{i-1}),\text{ for any }2 \leq i \leq |T|,
\end{equation} where the set $C_f(\cdot)$ is defined by the formula \eqref{CNHullDef}.

The property \eqref{THProperty} directly gives that $h_T \in \DCN{n}$. Due to Corollary \ref{ExtensionToCNCorr} and to Note \ref{ExtensionToCNNote} after Theorem \ref{ExtensionToCNTh}, the function $h_T$ can be extended to the function $f_T : r \cdot B_{\infty}^n \to \RR$ of the class $\CN{n}$. The set of all functions, obtained by this extension process, is denoted by $\mathscr{F}_{n,r}$. Additionally, Corollary \ref{NonSingularFunctionCriteriaCorr} states that the functions $h_T$ and $f_T$ are non-singular with respect to the problem $F$.

Let us show that the families $\mathscr{T}_{n,r}$ and $\mathscr{F}_{n,r}$ satisfy to the conditions of Lemma \ref{LowerBoundLm} and give a way to estimate the value of $\tau_F$.

\begin{theorem}
The inequality $\tau_F \geq 3^{n-1} \log_2 (2r-1)$ is true, where $F = \GenFun{\CN{n}}{r \cdot B_{\infty}^n \cap \ZZ^n}$. The same result is true for the class $\DCN{n}$.
\end{theorem}

\begin{proof}
The formula for $\mathscr{T}_{n,r}$ gives the equality $|\mathscr{T}_{n,r}| = (2r-1) |\mathscr{T}_{(n-1),r}|^3,$ and we have $|\mathscr{T}_{n,r}| = (2r-1)^{\frac{3^n-1}{2}}$. The claim is to show that the families $\mathscr{T}_{n,r}$ and $\mathscr{F}_{n,r}$ satisfy to the conditions of Lemma \ref{LowerBoundLm}. Assuming that it is true, we have the resulting inequalities:
\[
\tau_F \geq \log_2 |\mathscr{T}_{n,r}| \geq \frac{3^n - 1}{2} \log_2 (2r-1) \geq 3^{n-1} \log_2 (2r-1).
\]

Let $R$ be a resolving set for $f_T$ with respect to the problem $F$. We will show that the inclusion $T \subseteq R$ holds. Since the function $f_T$ is non-singular, the minimum point of $f_T$ is in $R$. The property \eqref{THProperty} gives a possibility to use Lemma \ref{SuccessiveMinAnalogLm}. Using this Lemma and the induction principle, we conclude that $T \subseteq R$ and the theorem follows.

The problem $G = \GenFun{\DCN{n}}{r \cdot B_{\infty}^n \cap \ZZ^n}$ is simpler than the problem $F$, because oracle calls on non-integral points are allowed for the problem $F$. Hence, the same estimate holds for $\tau_G$.
\end{proof}

\subsection{Lower bounds for the class $\ECN{n}$}

Let $r \geq 1$ and $F$ be the functional $\GenFun{\ECN{n}}{r \cdot B_{\infty}^n \cap \ZZ^n \setminus \{0\}}$ until the end of this subsection. The point $0$ is removed from the optimization domain, because it is a trivial minimum.

Analogously, we consider a finite family $\mathscr{T}_{n,r}$ of sets $T \subseteq (r \cdot B_{\infty}^n) \cap \ZZ^n$ and a family $\mathscr{H}_{n,r}$ of functions $h_T : T \to 0:(2^{n}-1)$. For any $T \in \mathscr{T}_{n,r}$, the function $h_T$ is a bijection between $T$ and $0:(2^{n}-1)$. The family $\mathscr{T}_{1,r}$ contains all $r - 1$ possible sets of the type $T = \{i,i+1\}$ for any $0 < i < r$. If $T = \{i,i+1\}$, then we put $h_T(0) = -1$, $h_T(\pm i) = 0$, and $h_T(\pm (i+1)) = 1$. All possible functions $h_T$, defined this way, form the family $\mathscr{H}_{1,r}$.

The family $\mathscr{T}_{n,r}$ is obtained from the family $\mathscr{T}_{n-1,r}$ in the following way. For any integral $i$, satisfying to the inequality $0 < i < r$, and for any pair $(T_1,T_2) \in \mathscr{T}^2_{n-1,r}$, we a construct new pair $(T_1[i], T_2[i+1])$ and put $T = T_1[i] \cup T_2[i+1]$. All possible sets $T$ that can be obtained in this way form the family$\mathscr{T}_{n,r}$. More formally,

\begin{equation}\label{TSetEvenDef}
\mathscr{T}_{n,r} = \bigcup\limits_{i = 1}^{r-1} \{T_1[i] \cup T_2[i+1] : \text{ for } (T_1,T_2) \in \mathscr{T}^2_{n-1,r}\},
\end{equation}
where $T[c] = \{ (x,c) : x \in T\}$.

For any $T \in \mathscr{T}_{n,r}$, the function $h_T : T \to \RR$ of the class $\mathscr{H}_{n,r}$ is defined in the following way. Due to the formula \eqref{TSetEvenDef}, we have $T = (T_1[i],T_2[i+1])$, for some pair $(T_1,T_2) \in \mathscr{T}^2_{n-1,r}$, and some value $i$, satisfying the inequality $0 < i < r$. Then $h_T(0) = -1$ and, for any $y \not= 0$,
\begin{equation}\label{HSetEvenDef}
h_T(\pm y) = \begin{cases}
h_{T_1}(x),\text{ for }y=(x,i),\\
2^{n-1} + h_{T_2}(x),\text{ for }y=(x,i+1),\\
\end{cases}
\end{equation}
where the functions $h_{T_k} \in \mathscr{H}_{n-1,r}$, for any $k \in \{1,2\}$, are defined inductively in the same way.

Let us consider a set $T \in \mathscr{T}_{n,r}$ and the function $h = h_T \in \mathscr{H}_{n,r}$. Let the sequence $x\I1,x\I2,\dots,x\I{|T|}$ be formed by the points of $T$, sorted by increasing values of the function $h_T$ on them. It follows from definition that the sequence $x\I1,x\I2,\dots,x\I{|T|}$ has the following property:
\begin{equation}\label{THPropertyEven}
x\I{i} \notin C_f(0,x\I1,x\I2,\dots,x\I{i-1}),\text{ for any }2 \leq i \leq |T|,
\end{equation} where the set $C_f(\cdot)$ is defined by the formula \eqref{CNHullDef}.

The property \eqref{THPropertyEven} directly gives that $h_T \in \EDCN{n}$. Due to Corollary \ref{ExtensionToCNCorr} and to Note \ref{ExtensionToCNNote} after Theorem \ref{ExtensionToCNTh}, the function $h_T$ can be extended to the function $f_T : r \cdot B_{\infty}^n \to \RR$ of the class $\ECN{n}$. The set of all functions, obtained by this extension process, is denoted by $\mathscr{F}_{n,r}$. Additionally, Corollary \ref{NonSingularFunctionCriteriaCorr} states that the functions $h_T$ and $f_T$ are non-singular with respect to the problem $F$.

Let us show that the families $\mathscr{T}_{n,r}$ and $\mathscr{F}_{n,r}$ satisfy to the conditions of Lemma \ref{LowerBoundLm} and give a way to estimate the value of $\tau_F$.

\begin{theorem}
The inequality $\tau_F \geq (2^n-1) \log_2 (r-1)$ is true, where $F = \GenFun{\ECN{n}}{r \cdot B_{\infty}^n \cap \ZZ^n \setminus \{0\}}$. The same result is true for the class $\EDCN{n}$.
\end{theorem}

\begin{proof}
The formula \eqref{TSetEvenDef} gives the recurrence relation $|\mathscr{T}_{n,r}| = (r-1) |\mathscr{T}_{(n-1),r}|^2,$ and we have $|\mathscr{T}_{n,r}| = (r-1)^{2^n-1}$. The claim is to show that families $\mathscr{T}_{n,r}$ and $\mathscr{F}_{n,r}$ satisfy to the conditions of Lemma \ref{LowerBoundLm}. The resulting estimate directly follows from it.

Let $R$ be the resolving set for $f_T$ with respect to the problem $F$. The claim is to prove the inclusion $T \subseteq R$. Since the function $f_T$ is non-singular, the minimum point of $f_T$ is in $R$. The property \eqref{THPropertyEven} gives a possibility to use Lemma \ref{SuccessiveMinAnalogLm}. Using this Lemma and the induction principle, we conclude that $T \subseteq R$ and the theorem follows.

The problem $G = \GenFun{\EDCN{n}}{r \cdot B_{\infty}^n \cap \ZZ^n \setminus \{0\}}$ is simpler than the problem $F$, because oracle calls on non-integral points are allowed for the problem $F$. Hence, the same estimate holds for $\tau_G$.
\end{proof}

\section{Minimization of a conic function in a fixed dimension}

In this section, we are going to construct an algorithm based on the comparison oracle for the conic function integer minimization problem. We assume that an optimal integer point contains in the ball $a + r \cdot B_2^n$, for some integral $r \geq 1$. In our work, this problem is denoted by $\GenFun{\CN{n}}{a + r \cdot B_2^n \cap \ZZ^n}$. For the sake of simplicity, we also assume that a minimized function $f \in \CN{n}$ is defined in every point of $\RR^n$, e.g. $\dom(f) = \RR^n$.

%------tut nado dobavit eche odin variant

Our algorithm uses ideas of seminal Lenstra's paper \cite{LEN}, as well as algorithms \cite{DADDIS,DPV11,HILDKOP,OERTELDIS,ICONV_MIN_REV}. Algorithms of this type are referred to as Lenstra's type algorithms. Our minimization procedure consists of two known ideas. The first idea is based on the concept of ``flatness'' from geometry of numbers that is also known as Khinchine theorem \cite{KHI}. If an initial ellipsoid has a sufficiently small width, e.g. it is flat by some direction, then we can slice the ellipsoid by relatively small amount of ellipsoids of a lower dimension along this direction. In the opposite case, when the initial ellipsoid has a sufficiently large width, it contains an integral point, and we can apply the second idea. The second idea is the cutting plane technique started from some initial ellipsoid containing an integral point, which gives us an ellipsoid of a lower volume that contains an integral point too. Yudin and Nemirovskii \cite{NEMIR,YUDIN} implemented this idea for the convex continuous function minimization problem, assuming that the $0$-th order oracle is given. We will apply the technique of Yudin and Nemirovskii for the comparison oracle and conic functions.

Further, we will describe important ideas from geometry of numbers, following \cite{HILDKOP}.

\subsection{Lattice Widths and the Shortest Vector Problem}

Finding flatness directions for branching on hyperplanes is the key technique of Lenstra's algorithm. To this end, we need to define the notion of a \emph{lattice width} of a convex set.

Let $P \subseteq \RR^n$ be a non-empty closed set and $c \in \RR^n$.
The \emph{width of $P$ along $c$} is the number
$$
\width_c(P) = \max\limits_{x \in P} c^\top x - \min\limits_{x \in P} c^\top x.
$$

The \emph{lattice width} of $P$ is defined as
$$
\width(P) = \min \{ \width_c(P) : c \in \ZZ^n \setminus \{0\} \},
$$ any $c$ that minimizes $\width(P)$ is called a \emph{flatness direction of $P$}. Clearly, flatness directions are invariant under translations and dilations.

\begin{theorem}[Khinchin's flatness theorem \cite{KHI}]\label{KHITh} Let $P \subseteq \RR^n$ be a convex body. Either $P$ contains an integer point, or $\width(P) \leq \omega(n)$, where $\omega(n)$ is a constant depending on the dimension only.
\end{theorem}

The currently best known bound for $\omega(n)$ is $O(n^{4/3} \log^c n)$ \cite{RUD} and it is conjectured that $\omega(n) = \Theta(n)$ \cite{BAN_LITVAK}. We will see next that, for the specific case of ellipsoids, we can obtain this bound.

We write ellipsoids in the form $\El(A,a) = \{x \in \RR^n : (x-a)^\top (A^{-1})^{\top} A^{-1} (x - a) \leq 1\} = \{x \in \RR^n : \|x-a\|_{A^\top A} \leq 1\}$, where $||x||_B = \sqrt{x^\top B^{-1} x}$, $A \in \RR^{n \times n}$ is a non-singular matrix and $a \in \RR^n$.

\begin{note}\label{DilationLm} Clearly, $B \El(A,a) = \El(B A, B^{-1} a)$ for any non singular $B \in \RR^{n \times n}$.

Let $c \in \ZZ^n$ be a flatness direction for $E = \El(A,0)$. Then, for any $\beta \in \RR$, $c$ is a flatness direction for $\frac{1}{\beta} E = \El(\frac{1}{\beta} A,0)$ with
$$
\frac{1}{\beta} \width(E) = \width(\El(\frac{1}{\beta} A,0)).
$$
\end{note}

\subsection{The Shortest lattice Vector Problem (SVP) and the Closest lattice Vector Problem (CVP)}

Let $A \in \QQ^{m \times n}$ and $a \in \QQ^n$, where $m,n$ are positive integers. The SVP and CVP for the $l_2$ norm can be formulated as follows, respectively:
$$
\min\limits_{x \in \Lambda(A) \setminus \{0\}} ||x||_2,
$$
$$
\min\limits_{x \in \Lambda(A)} ||x - a||_2,
$$ where $\Lambda(A) = \{A t: t \in \ZZ^n\}$ is the lattice induced by columns of the matrix $A$.

Due to the papers \cite{DINUR,MICC} the SVP and the CVP are hard to approximate within a constant factor and a factor $n^{c/ \log \log n}$, respectively. The first polynomial-time approximation algorithm for the SVP was proposed by Lenstra, Lenstra, and Lov\'asz in \cite{LLL}. Shortly afterwards, Fincke and Pohst in \cite{FP83,FP85}, Kannan in \cite{KAN83,KAN} described the first exact SVP and CVP solvers. Kannan's solver has a computational complexity of  $2^{O(n\,\log n)}$ in a dependence on the dimension $n$. The first SVP and CVP solvers that achieve the complexity $2^{O(n)}$ were proposed by Ajtai, Kumar, Sivakumar \cite{AJKSK01,AJKSK02}, Micciancio and Voulgaris \cite{MICCVOUL}. The previously discussed solvers are used for the Euclidean norm. Recent results for general norms are presented in \cite{BLNAEW09,DADDIS,DPV11,EIS11}. The paper of Hanrot, Pujol, Stehl\'e \cite{SVPSUR} gives a good survey and deep analysis about SVP and CVP solvers.

\begin{theorem}[Kannan \cite{SVPSUR,KAN}]\label{KannanSVPSolverTh} There are deterministic $n^{n/2 + o(n)} \poly(\size(A),\size(a))$-time and $\poly(n,\size(A),\size(a))$-space algorithms to solve the SVP and the CVP.
\end{theorem}

\begin{theorem}[Micciancio and Voulgaris \cite{SVPSUR,MICCVOUL}]\label{MV_SVPSolverTh} There are deterministic $2^{2n + o(n)} \poly(\size(A),\size(a))$-time and $2^{n + o(n)} \poly(\size(A),\size(a))$-space algorithms to solve the SVP and the CVP.
\end{theorem}

Kannan firstly observed that the SVP could be used to minimize the number of branching directions in his Lenstra's type algorithm \cite{KAN83,KAN}. We follow Eisenbrand in presenting this in the context of flatness directions \cite{50YEARS}.

\begin{note}\label{FlatnessDirectionComputeNote}
For an ellipsoid, a flatness direction can be computed by solving the SVP over the lattice $\Lambda(A^\top)$. To see this, consider the width along a direction $c$ of the ellipsoid $E = \El(A,0)$:
\begin{multline*}
\width_c(E) = \max\limits_{x \in E} c^\top x - \min\limits_{x \in E} c^\top x = \\
= \max\limits_{x \in B_{2}^n} c^\top A x - \min\limits_{x \in B_{2}^n} c^\top A x = 2 ||c^\top A||_2.
\end{multline*}
Finding the minimum lattice width is then the SVP over the lattice $\Lambda(A^\top)$.
\end{note}

\subsection{Results from geometry of numbers}

In this subsection, we again follow \cite{HILDKOP}. Geometry of numbers produces a small bound on the lattice width of an ellipsoid not containing an integer point. By considering our case of ellipsoids, we can produce an $O(n)$ bound. Using properties of LLL-reduced bases, Lenstra originally observed that this value is not exceed $2^{O(n^2)}$ \cite{LEN}. For an arbitrary lattice, the product of the length of a shortest vector in a lattice and the covering radius of the dual lattice is bounded by a constant $f(n)$ depending only on the dimension. Using the Fourier transform applied to a probability measure on a lattice, Banaszczyk showed that this function is bounded by $\frac{1}{2}n$.

\begin{theorem}[Banaszczyk \cite{BAN}]\label{BanLm}
Let $\Lambda \subset \RR^n$ be a lattice. Then $SV(\Lambda) \mu(\Lambda^*) \leq f(n) \leq \frac{1}{2} n$.
\end{theorem}

If we assume that a specific ellipsoid does not contain a lattice point, then the covering radius of the associated lattice is greater than one. Since the lattice width of an ellipsoid is simply twice the length of a shortest vector, we obtain the following inequality for ellipsoids.

\begin{theorem}[Eisenbrand \cite{50YEARS}]\label{EzLm}
If $E \subset \RR^n$ is an ellipsoid that does not contain an integer point, then $\width(E) \leq 2 f(n)$.
\end{theorem}

Thus a convenient bound follows directly from previous theorems.
\begin{corollary}[Hildebrand and K\"oppe \cite{HILDKOP}]\label{FlatEllipsoidLm}
Let $E \subset \RR^n$ be an ellipsoid that does not contain an integer point, then $\width(E) \leq n$.
\end{corollary}

The paper \cite{ICONV_MIN_REV} contains a very simple proof of the following lemma.
\begin{theorem}[Oertel \cite{ICONV_MIN_REV}]\label{SmallVolumeConvexSetLm}
Let $K \subset \RR^n$ be a bounded convex set. If $\vol(K) < 1$, then there is a translation $t \in \RR^n$, such that $(t + K) \cap \ZZ^n = \emptyset$.
\end{theorem}

Using results of Lemmas \ref{FlatEllipsoidLm} and \ref{SmallVolumeConvexSetLm} we have the following corollary.
\begin{corollary}\label{SmallVolumeEllipsoidLm}
Let $E \subset \RR^n$ be an ellipsoid and $\vol(E) < 1$, then $\width(E) \leq n$.
\end{corollary}

\subsection{Cuts in ellipsoids based on the comparison oracle of a conic function}

Starting from this moment, we follow \cite[P.\,342--348]{NEMIR} and \cite{YUDIN}. Let $M_I = M \cap \ZZ^n$, for any set $M \subseteq \RR^n$.

Let $a \in \RR^n$ and $||a||_2 = 1$, then the rotation cone around a ray $a$ with an angle $\phi$ is denoted by the symbol
$$
C(a,\phi) = \{x \in \RR^n : (x,a) \geq ||x||_2 \cos \phi\},\text{ for } 0 \leq \phi \leq \frac{\pi}{2}.
$$ A cone $C$ is said to be a \emph{$\phi$-angle cone}, if $C(a,\phi)$ is included to some translation of $C$, for some $a$.

\begin{lemma}[Yudin and Nemirovskii \cite{NEMIR}, p.\,345]\label{ConeCutLm}
Let $W = r \cdot B_2^n$, $a \in \RR^n$ and $||a||_2 = 1$.

If $\cos \phi < 1/n$, then the set $W \setminus C(a,\phi)$ can be included to an ellipsoid with the volume $\beta^n(\phi) \vol(W)$ and the center $-r \gamma(\phi) a$, where
$$
\gamma(\phi) = \frac{1 - n \cos \phi}{1 + n},
$$
$$
\beta(\phi) = 2 (\sin \frac{\phi}{2})^{\frac{n-1}{n}} (\cos \frac{\phi}{2})^{\frac{n+1}{n}} \cfrac{n (\frac{n-1}{n+1})^{\frac{1}{2n}}}{\sqrt{n^2-1}}.
$$
\end{lemma}

For $\phi = \phi_n = \arccos\left(\frac{1}{2n}\right)$ we have
$$
\beta(\phi_n) = 1 - \frac{d_n}{n^2},\quad d_n > 0,
$$
$$
\lim\limits_{n \to \infty} d_n = 1/8, \quad \gamma(\phi_n) = \frac{1}{2(n+1)}.
$$

\begin{note}[Yudin and Nemirovskii \cite{NEMIR}, p.\,345]\label{YNLemmaNote}
The proposition of Lemma \ref{ConeCutLm} is true if the value of $\beta(\phi_n)$ is changed to $\hat \beta(\phi_n) = \frac{1}{2}(1+\beta(\phi_n))$ and the cone $C(a,\phi_n)$ is moved to any position, such that its apex is included to a $\hat g(n) r$-neighborhood of the center of $W$.

It was also noticed in \cite{NEMIR} that $\hat g(n) \geq \frac{\hat c}{n}$, where $\hat c$ is absolute constant.
\end{note}

The proof of the following lemma is actually given in \cite[p.\,345]{NEMIR}, but we present a proof based on our notation.
\begin{lemma}\label{ConeCutConstructionLm}
Let $W = r \cdot B_2^n$ for some integral $r \geq 1$, $\phi_n = \arccos\left(\frac{1}{2n}\right)$, $f \in \CN{n}$ and $W \subseteq \dom(f)$. Then, there is a polynomial-time oracle-based algorithm that computes points $x\I{1},x\I{2},\dots,x\I{n+1} \in W$, such that a cone $C = \cone(x\I{1},\dots,x\I{n}|x\I{n+1})$ is a $\phi_n$-angle cone, $0 \notin C$ and $f(x) \geq f(0)$, for any $x \in C$.
\end{lemma}
\begin{proof}
Let $S$ be a regular simplex inscribed to $W$ and $s\I{1},s\I{2},\dots,s\I{n+1}$ be the vertices of $S$. Using a polynomial number of calls to the comparison oracle of $f$, we can find a maximal vertex of $S$. Suppose that it is $s\I1$. Let $p\I1 = s\I1$ be the apex of the regular pyramid $P_1$ defined as follows: $P_1$ has $n+1$ faces and vertices, the height of $P_1$ is collinear to the radius vector $p\I1$, the angles between the height and the faces are equal to $\phi$, if $v$ is vertex of $P_1$, then the radius vector $v$ is orthogonal to the edge $p\I1-v$. Let us suppose that the apex $p\I1$ of the pyramid $P_1$ has maximal value of the function $f$ between all vertices of $P_1$. Then, we output the cone $\cone(V|p\I1)$, where $V$ is the set of vertices of $P_1$ except $p\I1$. In the opposite case, let $p\I2 \not= p\I1$ be the vertex of the pyramid $P_1$ with the maximal value of $f$. In the next step of our iterative process, we build a regular pyramid $P_2$ with the apex $p\I2$ by the same rules as for $P_1$. The iterative process finishes at the moment, when the apex $p\I{k}$ of a pyramid $P_k$ becomes a vertex with a maximal value of the function $f$ between other vertices of $P_k$. After it we output the cone $\cone(V|p\I{k})$, where $V$ is the set of vertices of $P_k$ except $p\I{k}$.

Let us show that the process is finite. Definitely, by the construction we have that $||p\I{k}||_2 = \cos^{k}(\psi) ||p\I1||_2 = \cos^{k}(\psi)\,r$, where $\psi$ is the angle between the height and edges emerging from the apex $p\I{k}$ of the pyramid $P_k$. Clearly, the size of $\cos(\psi)$ polynomially depends on the size of $\cos(\phi_n) = \frac{1}{2n}$. Hence, after a polynomial on $n$ and $r$ number of steps we will have $||p\I{k}||_2 \leq \frac{1}{n} r$ and $p\I{k} \in S$. By Note \ref{CNSubsetQC}, the function $f$ is quasiconvex, so, $f(p\I{k}) \leq f(s\I1)=f(p\I1)$. The last inequality contradicts to the fact that the sequence $f(p\I{k})$ is strictly increasing. 

It is needed to note that faces of the pyramid $P_k$ can have irrational coefficients. So, we need to round them to rational values with a sufficient accuracy. It can be easily done by choosing the angle $\phi$ between the height and faces of $P_k$ slightly bigger than $\phi_n = \arccos\left(\frac{1}{2n}\right)$.

Let us show that the cone $C = \cone(V|p\I{k})$ satisfies to all of the required properties. Clearly, by construction, $C$ is $\phi_n$-angle cone and $0 \not \in C$, because the point $0$ is always included in the cone spanned by edges of the pyramid $P_k$, for each $k$. Let us show that $f(x) > f(0)$, for any $x \in C$. We can assume that $k > 1$, because in the opposite case the property is trivial by the quasiconvexity of $f$. Since $k > 1$, we have $p\I{k} \notin S$. By equivalent definition of the class $\CN{n}$ from Theorem \ref{CNEquivalentDefinitions}, we have $f(x) \geq f(p\I{k})$, for any $x \in C$. Since $f(p\I{k}) > f(p\I{1}) = f(s\I{1})$, we have $f(p\I{k}) > f(0)$, by the quasiconvexity of $f$.

Figure 3 is an illustration of the first three steps of this construction, when the pyramids $P_1$, $P_2$, $P_3$ are constructed. It can be shown that after two additional steps the final pyramid will be included in $S$.

\begin{figure}[ht!]\label{ConeCutConstructionFig}
\centering
\includegraphics[clip,scale=0.35]{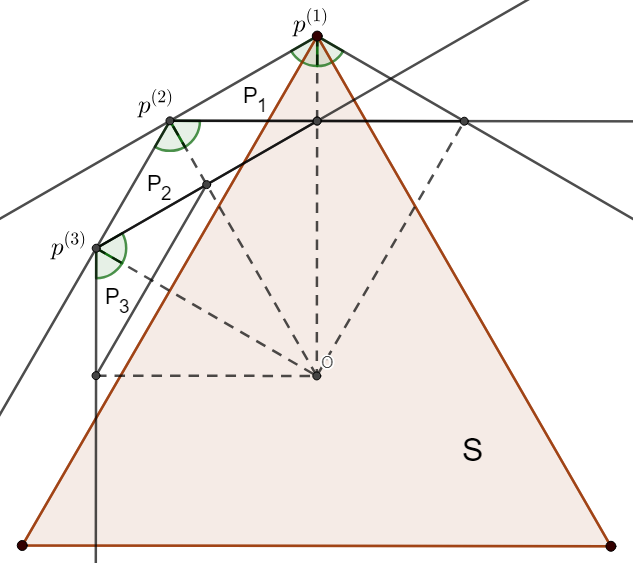}
\caption{An example for construction of the pyramids $P_1,P_2,P_3$. The apexes are $p\I1,p\I2, p\I3$, respectively, the initial simplex is $S$ (red), the angle between heights and faces of the pyramids is chosen to be equal $\pi/3$.}
\end{figure}

\end{proof}

Lemmas \ref{ConeCutLm}, \ref{ConeCutConstructionLm} give us the main tool to construct ellipsoids of a lower volume.

\begin{corollary}\label{ConeCutConstructionCor}
Let $W = r \cdot B_2^n$ for some integral $r \geq 1$, $f \in \CN{n}$ and $W \subseteq \dom(f)$. Let, additionally, $z \in \ZZ^n$ and $||z||_2 \leq \frac{\hat c}{2 n} r$.  Then there is a polynomial-time comparison oracle-based algorithm to construct an ellipsoid $E$ with the following properties:

1) $\vol(E) = {\hat \beta(\phi_n)}^n \vol(W)$, where the values $\phi_n$, $\hat \beta(\phi_n)$, $\hat c$ are defined after Lemma \ref{ConeCutLm};

2) $E \cap \RMIN{1}{W_I}{f} \not= \emptyset$.
\end{corollary}
\begin{proof}
Consider a ball $G = z + \frac{\hat c}{2 n} r \cdot B_2^n$. Clearly, $G \subseteq \frac{\hat c}{n} r \cdot B_2^n$. Using Lemma \ref{ConeCutConstructionLm}, we construct a $\phi_n$-angle cone $C = \cone(x\I{1},\dots,x\I{n}|x\I{n+1})$, such that $x\I{1},x\I{2},\dots,x\I{n+1} \in G$, $z \notin C$ and $f(x) \geq f(z)$ for any $x \in C$. By Lemma \ref{ConeCutLm} and Note \ref{YNLemmaNote} after it, we can inscribe the set $W \setminus C$ into the desired ellipsoid $E$.
\end{proof}

Finally, before we give the main minimization algorithm we describe a preprocessing procedure that will be very helpful, when we need to reduce the dimension of an initial problem and to find a short lattice basis in the reduced space. 
Here we fully follow Dadush' IP-Preprocessing Algorithm \cite[pp.\,223--225]{DADDIS}.

\begin{lemma}\label{PreprocessingAlg} Let $L$ be an $n$-dimensional lattice given by a basis $B \in \QQ^{n \times n}$, and $H = \{x \in \RR^n : A x = b\}$ be an affine subspace, where $A \in \QQ^{m \times n}$ and $b \in \QQ^m$. 
Let also $E = a_0 + r \cdot B_2^n$, where $a_0 \in \QQ^n$ and $r \in \QQ_+$. Then, there is an algorithm with the bit complexity $2^{O(n \log n)} \poly(s)$, where $s$ is input size, which either decides that $E \cap L \cap H = \emptyset$ or returns
\begin{enumerate}
\item[\rm 1)] a shift $p \in L$,
\item[\rm 2)] a sublattice $L' \subseteq L$, $\dim L' = k \leq n$, given by a basis $b_1',\dots,b_k'$,
\item[\rm 3)] a vector $a_0' \in \linh L'$ and a radius $r'$, 
      $0 < r' \leq r$,
\end{enumerate} 
satisfying the following properties
\begin{enumerate}
\item[\rm 1)] $E \cap L \cap H = (E' \cap L') + p$, where 
$E' = \{x \in \linh L' :~ \|x - a_0'\|_2 \leq r' \}$, and $a_0'$ is the orthogonal projection of $a_0$ into $H$,
\item[\rm 2)] $\max\limits_{1 \leq i \leq k} \|b_i'\| \leq 2 \sqrt{k} r'$,
\item[\rm 3)] $a_0'$, $\rho'$, $b_1', \dots, b_k'$ and $p$ have polynomial in $s$ encodings.
\end{enumerate}
\end{lemma}
\begin{proof}
See~\cite[pp.\,223--225]{DADDIS}.
\end{proof}

\subsection{The conic function integer minimization algorithm}

\begin{theorem}\label{ConicMinTh}
Let $F = \GenFun{\CN{n}}{r_0 \cdot B_2^n \cap \ZZ^n}$, for some integral $r_0 \geq 1$ and $f \in \CN{n}$ be a function defined everywhere on $\RR^n$. Then the problem $F(f)$ can be solved by an algorithm with the bit-complexity $T_{bit}(n,r_0) = 2^{O(n)} n^{2 n} \log r_0$ and the oracle-complexity $T_{oracle}(n,r_0) = (2 n^2)^{n + O(1)} \log r_0$. The space complexity of the algorithm is $2^{n + o(n)} \poly(\log r_0)$.
\end{theorem}

\begin{proof}
Consider the following algorithm:
\begin{algorithmic}[1]
\REQUIRE The comparison oracle for a function $f$; a lattice $L \subseteq \ZZ^n$; a point $a \in \QQ^n$ and a radius $r \in \QQ_+$; a rational affine subspace $H$.
\ENSURE Return EMPTY if the set $(a + r \cdot B_2^n) \cap L \cap H$ is empty. If it is not, return the minimum point of $f$ in the set $(a + r \cdot B_2^n) \cap L \cap H$.
\STATE $(a,r,L,p) := \text{Preprocessing}(a,r,L,H)$.
\STATE Set $p$ as the origin, when we call the comparison oracle for $f$. In other words, we put $f(x) := f(x - p)$.
\STATE Set $E := \{ x \in \linh L : \|x - a\|_2 \leq r\}$, $n := \dim L$.
\REPEAT
\STATE {\bf Construction of a scaled ellipsoid.} Assuming that $E = \El(A,a)$, construct a scaled ellipsoid $\hat E = \frac{\hat c}{2 n} E = \El(\frac{\hat c}{2 n} A,a)$.

\STATE {\bf Computing width and flat direction of $\hat E$.} Compute the width $w$ and a flat direction $c \in L$ of the ellipsoid $\hat E$ using Note \ref{FlatnessDirectionComputeNote} and Theorem \ref{MV_SVPSolverTh}. And set $\alpha := \max\limits_{x \in \hat E} c^\top x$ and $\beta := \min\limits_{x \in \hat E} c^\top x$.

\IF{$w > n$}
\STATE {\bf Find an integral point inside $\hat E$.} Compute $z$ as a solution of the CVP in the lattice $L$ with respect to the norm $\|\cdot\|_{A^\top A}$ using Theorem \ref{MV_SVPSolverTh}. Since $w > n$, then, by Corollary \ref{FlatEllipsoidLm}, $\hat E \cap L \not= \emptyset$, and we have $z \in \hat E$.

\STATE {\bf Construct an ellipsoid of a lower volume than $E$.} After the map $x \to A x$ we have $E \to A a + B_2^n$, $\hat E \to A a + \frac{\hat c}{2 n} \cdot B_2^n$, $z \to A z$ and the comparison oracle of the function $f(x)$ transforms to an oracle for the function $f(A x)$. Applying Corollary \ref{ConeCutConstructionCor} to the ball $A a + B_2^n$ and the point $A z$, we construct an ellipsoid $W$ of the volume ${\hat \beta}^n(\phi_n) \vol(B_2^n)$ that contains the point $A z$. Suppose that $W = \El(B,b)$, for $B \in \QQ^{n \times n}$ and $b \in \QQ^n$. After the reverse transform $x \to A^{-1} x$ and $W \to A W$, we have the resulting ellipsoid $W = \El(A B,A^{-1} b)$ of a lower volume than $E$ that contains the integral point $z$.

\STATE $E := W$.
\ENDIF

\UNTIL{$w > n$}.

\FORALL{$t \in \{ c^\top x :~ x \in E\} \cap \ZZ$} \label{enumeration_step}
\STATE $H_t := \{x \in \linh L:~ c^\top x = t\}$.
\STATE Make the recursive call of the algorithm with the input parameters $(a,r,L,H_t)$.
\ENDFOR
\RETURN If all recursive calls of the algorithm have returned EMPTY, then return EMPTY. In the opposite case, return $p+y$, where $y$ is a minimum point of $f$ between all recursive calls. 
\end{algorithmic}

To solve the initial problem, we need to run this algorithm with the input parameters $r = r_0$, $a = 0$, $L = \ZZ^n$, $H = \{x \in \RR^n : 0^\top x = 0 \}$ and the comparison oracle for the function $f$.

The algorithm is correct due to the following invariant statements:

1) Each time, when we construct an ellipsoid $E$ of a lower volume in Line 9, we always have $E \cap \RMIN{1}{B_I}{f} \not= \emptyset$, due to Lemma \ref{ConeCutConstructionCor};

2) In Line 13, if $x \in W \cap L$, then $c^\top x = k$, for some $k \in \ZZ$, such that $\lceil \beta \frac{2 n}{\hat c} \rceil \leq k \leq \lfloor \alpha \frac{2 n}{\hat c} \rfloor$. The last fact follows from the lattice width definition.
%Compute a matrix $B \in \ZZ^{n \times (n-1)}$, such that $c \ZZ + \Lambda(B) = \ZZ^n$.
% 
%
%Let us discuss the data encoding in the algorithm. For the ellipsoid $E = \El(A,a)$ we encode the matrix $A$ in the form $A = \alpha \bar A$, where $\alpha \in \QQ$ and $\bar A \in \ZZ^{n \times n}$, the same encoding is true for the vector $a$. Then, $\size(A) = \size(\alpha) + n^2 + \sum_{i,j} \lceil \log_2 (1 + |\bar A_{i,j}|) \rceil$, where $\size(\alpha) = \lceil \log_2 (1 + |p|) \rceil + \lceil \log_2 (1 + |q|) \rceil$, for $\alpha = \frac{p}{q}$. Using this encoding, it can be shown that the size of the matrix $A$ and the vector $a$ in the lines 7-8,\,12 increases only on an additive factor $\poly(n)$ and some constant multiplicative factor. We will show that the algorithm has $\poly(n) \log v$ iterations in the lines $1-9$, so, we have $s^\prime = \poly(n) \log v + O(s)$, where $s = \size(A) + \size(a)$ is the initial ellipsoid size and $s^\prime$ is the size before recursion in the line 13. 

Let $r$, $a$, $L$ and $H$ be input parameters of some recursive call of the algorithm.

Consider the oracle-complexity and the iterations number of the Repeat-Until loop in Lines 7-11. Let $v$ be a volume of the initial ellipsoid $E$ in Line 3. By Lemma \ref{PreprocessingAlg}, we have $r \leq r_0$. Clearly, $v \leq r_0^n \vol B_2^n$. Due to Lemma \ref{ConeCutConstructionCor}, the volume of the ellipsoid $E$ decreases with a speed of a geometric progression. Hence, after at most $\poly(n) \log v$ iterations we will have $\vol(E) < 1$. Due to Lemma \ref{SmallVolumeEllipsoidLm}, it gives that $\width(E) \leq n$. So, the cycle in the lines 1-9 has at most $\poly(n) \log v = \poly(n) \log r_0$ iterations and the same number of calls to the oracle of $f$. Therefore, the number of iterations and the total oracle-complexity of the cycle is $\poly(n) \log r_0$.

The cycle in Steps 11-13 consists of at most $\frac{2 n^2}{\hat c}$ recursive calls of the same algorithm. Clearly, the algorithm oracle complexity depends only from parameters $n$ and $r$. Then, we have
$$
\hat T_{oracle}(n,r_0) \leq \poly(n) \log r_0 + \frac{2 n^2}{\hat c} \hat T_{oracle}(n-1,r_0).
$$ Hence, $T_{oracle}(n,r_0) = (2 n^2)^n \poly(n) \log r_0$.

Let us estimate the algorithm bit-complexity. By Lemma \ref{PreprocessingAlg}, we always have $r \leq r_0$ and $\size(L) = \poly(n, \log r_0)$, and we can choose $r$, such that $\size(r) = O(\log r_0)$. Let us show, that $\size(H) = \poly(n, \log r_0)$. Definitely, for each recursive call we have $H = \{x \in \RR^n : c^\top x = t\}$, where $t \in [\alpha, \beta] \cap \ZZ$ and $c$ is a flat direction of the ellipsoid $E$ from previous recursive call. Let $E = \El(A,a)$. Doing the same analysis as in \cite[p. 99--101]{GRLOVSCH}, and noting that the loop in Lines 7-11 has $\poly(n) \log r_0$ iterations, it can be shown that $\size(A) = \poly(n,\log r_0)$. Since $c$ is a flat direction of $E$, we also have $\size(c) = \poly(n,\log r_0)$. Clearly, the hyperplane $c^\top x = t^\prime$ intersects the initial ball $B = \{x \in \linh(L) : \|x - a\|_2 \leq r\}$ for some $t^\prime \in [\alpha, \beta]$. Since the ellipsoid $E$ is flat after Line 11, we have $\alpha - \beta \leq \frac{2 n^2}{\hat c}$, and hence $$\max\{|\alpha|, |\beta|\} \leq \max\limits_{x \in B} c^\top x + \frac{2n^2}{\hat c}.$$ So the sizes of $\alpha$ and $\beta$ are polynomial by $n$ and $\log r_0$, and consequently $\size(H) = \poly(n, \log r_0)$. Finally, we need to show that the size of the parameter $a$ is also polynomial by $n$ and $\log r_0$, but it easily follows from the fact that $a$ is an orthogonal projection of the point $0$ on the affine space induced by flat directions, generated during recursive calls of the algorithm.

Let $s$ be the summary size of the input parameters $r$, $a$, $L$ and $H$. It has been already proven that $s = \poly(n, \log r_0)$. The operations is Lines 1,5,9 can be done in $\poly(n,s)$ time. Due to Note \ref{FlatnessDirectionComputeNote} and to Theorem \ref{MV_SVPSolverTh}, the complexity of steps 6,8 is equivalent to the complexity of solving the SVP and the CVP problems, which is $2^{O(n)} \poly(s)$. Therefore, the total bit-complexity of the loop in lines 4-11 is $2^{O(n)} \poly(\log r_0)$.

Clearly, the algorithm bit-complexity depends only from parameters $n$ and $r$. Finally, we have
$$
\hat T_{bit}(n,r_0) \leq 2^{O(n)} \poly(\log r_0) + \frac{2 n^2}{\hat c} \hat T_{bit}(n-1,r_0).
$$ Hence, $T_{bit}(n,r_0) = 2^{O(n)} n^{2 n} \poly(\log r_0)$.

\end{proof}

\begin{note}
If it is critical to have a polynomial space-complexity constraint to solve the considered problem, then we can use the Kannan's SVP and CVP solvers \cite{KAN83,KAN} instead of the solvers of Micciancio and Voulgaris \cite{MICCVOUL}, see Theorems \ref{KannanSVPSolverTh}, \ref{MV_SVPSolverTh}. It gives
$$
T_{bit}(n,r) = n^{2.5 n + o(n)} \poly(\log r),
$$ the oracle-complexity $T_{oracle}(n,r)$ states the same.
\end{note}

\subsection{Examples of concrete problems that can be expressed by conic functions}

In this section, we show that integer minimization of a quasiconvex polynomial with quasiconvex polynomial constraints can be expressed by the language of conic functions. Using the result of Theorem \ref{ConicMinTh}, the last fact repeats the main result of the work \cite{HILDKOP} of Hildebrand and K\"oppe.

Consider the problem
\begin{align}
&f(x) \to \min\label{ConstraintPolyMinProblem}\\
&\begin{cases}
g_i(x) \leq 0,\text{ for }i \in 1:m, \\
x \in r \cdot B_2^n \cap \ZZ^n,
\end{cases}\notag
\end{align} where $f$ and $g_i$ be quasiconvex polynomials. It has been shown (see the problem \eqref{ConditionalMinProblem}) that this problem is equivalent to the problem $\GenFun{h(x)}{r \cdot B_2^n \cap \ZZ^n}$, where $h(x) = (\,(t(x))_+, f(x)\,)$ and $t(x) = \max \{ g_i(x): 1 \leq i \leq m \}$. The complexity of the lexicographical order oracle for the function $h(x)$ is $O(m\,d\,M\,\poly(n, \log r))$, where $d$ and $M$ are the maximal degree and the number of monomials in a sparse encoding of the polynomials respectively. Using the Theorem \ref{ConicMinTh}, we have an algorithm for the problem \eqref{ConstraintPolyMinProblem} with bit-complexity $2^{O(n)} n^{2 n}\,m\,d\,M\,(\log r)^{O(1)}$, which repeats the main result of the paper \cite{HILDKOP}.

Additionally, our tools can be helpful to design FPT-algorithms for some combinatorial optimization problems. See papers \cite{FIXEDP,GAVKNOP} for details.

Let us present another example of a problem that can be expressed using this language. Let $a,b$ be two positive integers, the problem to compute Greatest Common Divisor (GCD) of two integers can be formulated as follows:
\begin{align}
&|a x_1 - b x_2| \to \min\label{GCDProblem}\\
&\begin{cases}
x \in \ZZ^2 \setminus \{0\}.
\end{cases}\notag
\end{align} Clearly, the optimal point of this problem contains in the ball of the radius $r = \sqrt{a^2 + b^2}$. Since $f(x) = |a x_1 - b x_2|$ is an even conic function, the GCD problem is equivalent to the even conic function minimization problem. The paper \cite{2DIMMIN} contains an algorithm for such problems in the dimension $2$ based on calls to the $0$-th order oracle with the orcle-based complexity be $O(\log r)$. It can be shown that the algorithm of the paper \cite{2DIMMIN}, applied to the GCD problem, give us complexity $O(s^2)$, for $s$ be binary encoding length of input, which matches the Euclid's algorithm complexity.

{\bf Acknowledgments}

This work was supported by the Russian Science Foundation Grant No. 17-11-01336.

\end{document}